\documentclass [11pt]{amsart}
\usepackage {amsmath, amssymb, amscd, mathrsfs, tikz-cd, comment, stmaryrd, hyperref,  bbm, enumerate, url, graphicx, color}
\usepackage[all, cmtip]{xy}
\usepackage[text={6.5in,9in},centering,letterpaper,dvips]{geometry}

\newtheorem {theorem}{Theorem}[section]
\newtheorem {lemma}[theorem]{Lemma}
\newtheorem {proposition}[theorem]{Proposition}
\newtheorem {corollary}[theorem]{Corollary}

\newtheorem {definition}[theorem]{Definition}
\newtheorem {question}[theorem]{Question}
\theoremstyle{remark}
\newtheorem {fact}[theorem]{Fact}
\newtheorem {remark}[theorem]{Remark}

\newtheorem {example}[theorem]{Example}

\def\Z {{\mathbb{Z}}}
\def\N {{\mathbb{N}}}
\def\R {{\mathbb{R}}}

\def\CP{\mathbb{CP}}
\def\tr{\operatorname{tr}}
\def\Tr{\operatorname{Tr}}
\def\bPsi{\underline{\Psi}}
\def\Kh{\operatorname{Kh}}
\def\KhRN{\operatorname{KhR}_N}

\def\KhR{\operatorname{KhR}}

\def\S{\mathcal{S}^N}

\def\Sz{\mathcal{S}_0^N}
\def\Sinf{\mathcal{S}_\infty^N}
\newcommand{\sunderline}[1]{\underline{#1\mkern-3mu}\mkern3mu }
\def\cSz{\sunderline{\mathcal{S}}_0^{N, \alpha}}
\def\cS{\sunderline{\mathcal{S}}_0^{N}}
\def\Sztwo{\mathcal{S}_0^2}

\def\gl{\mathfrak{gl}}
\def\glN{\gl_N}

\def\del{\partial}
\def\im{\operatorname{im}}

\def\k{\mathbbm{k}}

\def\Aut{\operatorname{Aut}}
\def\Hom{\operatorname{Hom}}
\def\End{\operatorname{End}}

\def\RW{H^{*, *}_{\operatorname{RW}}}
\newcommand{\CC}{\mathcal{C}}
\newcommand{\DD}{\mathcal{D}}
\newcommand{\HHz}{\mathrm{HH}_0}
\def\HH{\operatorname{HH}}

\def\Id{\operatorname{Id}}

\def\ap{\langle \alpha' \rangle}

\newcommand{\Tang}{\boldsymbol{\mathrm{Tang}}}
\newcommand{\Foam}{\boldsymbol{\mathrm{Foam}}_N}
\newcommand{\Foamdg}{\Foam^{\mathrm{dg}}}
\newcommand{\Foamtwo}{\boldsymbol{\mathrm{Foam}}_2}
\newcommand{\Foamtwodg}{\Foamtwo^{\mathrm{dg}}}
\newcommand{\KhRbracket}[1]{\left\llbracket #1 \right\rrbracket}
\newcommand{\Ta}{\boldsymbol{\mathrm{T}}}

\newcommand{\Ch}{\operatorname{Ch}}
\newcommand{\Chdg}{\Ch_{\mathrm{dg}}}
\newcommand{\modgrfr}{\mathrm{-mod}^{\mathrm{gr.fr.}}}

\newcommand{\arxiv}[1]{\href{https://arxiv.org/abs/#1}{\small  arXiv:#1}}
\newcommand{\doi}[1]{\href{https://dx.doi.org/#1}{{\small DOI:#1}}}

\newcommand{\googlebooks}[1]{(preview at \href{https://books.google.com/books?id=#1}{google books})}

\newcommand{\numdam}[1]{}

\newcommand{\zeroh}{0}
\newcommand{\oneh}{1}
\newcommand{\twoh}{2}
\newcommand{\threeh}{3}
\newcommand{\fourh}{4}
\newcommand{\fourm}{four}

\newcommand{\Wone}{W_1} % 1-handlebody
\newcommand{\Wtwo}{W_2} % 2-handlebody
\newcommand{\Wthree}{W_3} % 3-handlebody
\newcommand{\Wfour}{W_4} % 4-handlebody
\newcommand{\Wgeneric}{W} % generic 4-manifold
\newcommand{\Whandle}{W'} % generic 4-manifold with handle attached

\usepackage{tikz}
\usepackage{tikz-cd}
\usetikzlibrary{math}
\usetikzlibrary{decorations.markings}
\usetikzlibrary{decorations.pathreplacing}
\usetikzlibrary{arrows,shapes,positioning}
\tikzstyle directed=[postaction={decorate,decoration={markings,
    mark=at position #1 with {\arrow{>}}}}]
\tikzstyle rdirected=[postaction={decorate,decoration={markings,
    mark=at position #1 with {\arrow{<}}}}]

\tikzset{anchorbase/.style={baseline={([yshift=-0.5ex]current bounding box.center)}},
    tinynodes/.style={font=\tiny,text height=0.75ex,text depth=0.15ex},
    smallnodes/.style={font=\scriptsize,text height=0.75ex,text depth=0.15ex},
    >={Latex[length=1mm, width=1.5mm]}
  }
  \tikzset{
    partial ellipse/.style args={#1:#2:#3}{
        insert path={+ (#1:#3) arc (#1:#2:#3)}
    }
}
\tikzcdset{arrow style=tikz, diagrams={>=stealth}}

  \newcommand{\pu}{to [out=90,in=270]}

  \newcommand{\fsphere}[3]{
\draw[white, line width=1mm] (#1-#3,#2) to [out=270,in=180] (#1,#2-#3*2/3) to [out=0,in=270] (#1+#3,#2);
\draw (#1-#3,#2) to [out=270,in=180] (#1,#2-#3*2/3) to [out=0,in=270] (#1+#3,#2);
\draw[thick] (#1,#2) circle (#3);
}

\newcommand{\bsphere}[3]{
\draw[dashed] (#1-#3,#2) to [out=90,in=180] (#1,#2+#3*2/3) to [out=0,in=90] (#1+#3,#2);
}

\newcommand{\fgsphere}[3]{
\draw[white, line width=1mm] (#1-#3,#2) to [out=270,in=180] (#1,#2-#3*2/3) to [out=0,in=270] (#1+#3,#2);
\draw[opacity=.3] (#1-#3,#2) to [out=270,in=180] (#1,#2-#3*2/3) to [out=0,in=270] (#1+#3,#2);
\draw[thick, opacity=.3] (#1,#2) circle (#3);
}

\newcommand{\bgsphere}[3]{
\draw[dashed,opacity=.3] (#1-#3,#2) to [out=90,in=180] (#1,#2+#3*2/3) to [out=0,in=90] (#1+#3,#2);
}

\newcommand{\sphere}[3]{
\bsphere{#1}{#2}{#3}
\fsphere{#1}{#2}{#3}
}

\newcommand{\opentrefoil}[5]{
\begin{scope}[shift={(#1,#2-5.9)},yscale=#3,xscale=#4]
\draw[#5] (-0.3,5.15) to [out=150,in=270] (-1,6) \pu (-1,8) (0,8.5) to [out=270,in=110] (0.85,6.3)
(-0.7,6) to [out=20,in=160] (1,6) to   [out=340,in=60] (1.4,5) to [out=240,in=330] (0.3,4.85)
(0.85,5.7) to [out=240,in=30] (0,5) to   [out=210,in=300] (-1.4,5) to [out=120,in=200] (-1.15,5.85);
\end{scope}
}

\newcommand{\bentline}[3]{
\draw[#3] (-7*#1,2*#1+#2) to [out=30,in=180] (0*#1,3*#1+#2) to [out=0,in=150] (7*#1,2*#1+#2);
\draw[#3,dotted] (-8*#1,1.5*#1+#2) to [out=30,in=210] (-7*#1,2*#1+#2)  (8*#1,1.5*#1+#2) to [out=150,in=330] (7*#1,2*#1+#2);
}
\newcommand{\bentlinegap}[4]{
\draw[#3] (-7*#1,2*#1+#2) to [out=30,in=180] (0*#1,3*#1+#2) (7*#1+#4,3*#1+#2) to [out=0,in=150] (14*#1+#4,2*#1+#2);
\draw[#3,dotted] (-8*#1,1.5*#1+#2) to [out=30,in=210] (-7*#1,2*#1+#2) (15*#1+#4,1.5*#1+#2) to [out=150,in=330] (14*#1+#4,2*#1+#2);
}
\newcommand{\bentlinegaps}[4]{
\draw[#3] (-7*#1,2*#1+#2) to [out=30,in=180] (-1*#1,3*#1+#2) (8*#1+#4,3*#1+#2) to [out=0,in=150] (14*#1+#4,2*#1+#2);
\draw[#3,dotted] (-8*#1,1.5*#1+#2) to [out=30,in=210] (-7*#1,2*#1+#2) (15*#1+#4,1.5*#1+#2) to [out=150,in=330] (14*#1+#4,2*#1+#2);
}

\definecolor{kwcolor}{rgb}{.05, .5, .3}
\definecolor{pwcolor}{rgb}{.15, .5, .5}
\definecolor{cmcolor}{rgb}{.8, 0, .05}

\begin{document}
%!TEX root=skein.tex
\title{Skein lasagna modules and handle decompositions}

\author[Ciprian Manolescu]{Ciprian Manolescu}
\thanks{CM was supported by NSF Grant DMS-2003488 and a Simons Investigator Award.}
\address{Department of Mathematics, Stanford University, Stanford, CA 94305, USA}
\email{\href{mailto:cm5@stanford.edu}{cm5@stanford.edu}}

\author[Kevin Walker]{Kevin Walker}
\address{Microsoft Station Q, Santa Barbara, CA 93106, USA}
\email{\href{mailto:kevin@canyon23.net}{kevin@canyon23.net}}

\author[Paul Wedrich]{Paul Wedrich}
\thanks{PW acknowledges support by the Deutsche Forschungsgemeinschaft (DFG, German Research Foundation) under Germany's Excellence Strategy - EXC 2121 ``Quantum Universe'' - 390833306.}
\address{Fachbereich Mathematik, Universit\"at Hamburg, Bundesstra{\ss}e 55,
20146 Hamburg, Germany}
\email{\href{mailto:paul.wedrich@uni-hamburg.de}{paul.wedrich@uni-hamburg.de}}

%\subjclass[2020]{\href{http://www.ams.org/mathscinet/search/mscdoc.html?code=57M25}{57M25}}

%\keywords{}

\begin{abstract}
The skein lasagna module is an extension of Khovanov--Rozansky homology to the
setting of a four-manifold and a link in its boundary. This invariant plays the
role of the Hilbert space of an associated fully extended
$(4+\epsilon)$-dimensional TQFT. We give a general procedure for expressing the
skein lasagna module in terms of a handle decomposition for the four-manifold.
We use this to calculate a few examples, and show that the skein
lasagna module can sometimes be locally infinite dimensional.
\end{abstract}
\maketitle
%\vspace{-1cm}
%\tableofcontents

\section{Introduction}
Homological invariants such as Khovanov homology \cite{Kh} and Khovanov-Rozansky
homology \cite{KhR} are at the center of modern knot theory. These invariants
were originally defined for links in $\R^3$. Extending them to links in
arbitrary $3$-manifolds is a problem that garnered much attention recently, from
various perspectives (categorification at roots of unity \cite{KhQ, QE, QRSW},
theoretical physics \cite{WittenKh, GPV, GPPV}, etc.)

One such extension was introduced in \cite{MWW}, based on higher category theory
and using the concept of blob homology \cite{MWblob}. Given a smooth, oriented,
compact \fourm-manifold $\Wgeneric$ and a framed oriented link $L$ in the boundary $\del \Wgeneric$,
the construction in \cite{MWW} associates to the pair $(\Wgeneric, L)$ a 
homology theory graded by $\Z^3 \times H_2(\Wgeneric, L; \Z)$ and denoted $\S_*(\Wgeneric; L)$. One of the three integer gradings is called the blob
degree, and for our purposes we will focus on the theory in blob degree zero,
$\Sz(\Wgeneric, L)$. This is called the {\em skein lasagna module} of $(\Wgeneric, L)$ and has a
relatively simple definition, reminiscent of the definition of the skein module
of a $3$-manifold. The skein lasagna module is defined as the span of the {\em
lasagna fillings} of $\Wgeneric$ with boundary $L$, modulo an equivalence relation. The
lasagna fillings are certain decorated surfaces connecting $L$ to other links in
the boundaries of $4$-balls inside $\Wgeneric$, and the equivalences come from cobordism
maps in Khovanov-Rozansky homology.

Skein lasagna modules are challenging to compute. It was proved in \cite{MWW}
that when $\Wgeneric=B^4$, the invariant $\Sz(B^4; L)$ coincides with the
Khovanov-Rozansky homology of the link $L$. Further computational methods were
developed in \cite{MN}, with a focus on \twoh-handlebodies (four-manifolds
obtained from $B^4$ by attaching \twoh-handles). This allowed the calculation of
the skein lasagna modules (in some gradings) for \fourm-manifolds such as the
complex projective plane, and disk bundles over $S^2$.

In this paper, building on the work in \cite{MWW} and \cite{MN}, we give a new
formula for the skein lasagna module of a link in the boundary of an arbitrary
four-manifold. We start by choosing a handle decomposition for the
four-manifold. For simplicity, we may assume that we have a single \zeroh-handle. We
then study how the skein lasagna module changes under adding handles. Disjoint
unions, \fourh-handles and many cases of \twoh-handles were already studied in
\cite{MN}, so the main thing left is to understand \oneh- and \threeh-handles. 

With regard to \threeh-handles, we have the following:
\begin{theorem} 
\label{Thm:3h}
Suppose that we have a four-manifold $\Wgeneric$ with boundary $Y$, and let $\Whandle$ be the
result of attaching a \threeh-handle to $\Wgeneric$ along a sphere $S \subset Y$. Let also $L$ be a
framed link in $Y$ disjoint from $S$, and $L'$ the corresponding link in $\del \Whandle$.
The equator $J$ of $S$ splits the sphere into two hemispheres, each of which
induces a cobordism map from $\Sz(\Wgeneric; L \cup J)$ to $\Sz(\Wgeneric; L)$. Then, the skein
lasagna module $\Sz(\Whandle; L')$ is isomorphic to the coequalizer of these two
cobordism maps. (See Theorem~\ref{thm:3h} for a more precise statement.) 
\end{theorem}

Next, we combine Theorem~\ref{Thm:3h} with the treatment of \twoh-handles in
\cite{MN} to get a general result, reducing the calculation of the skein lasagna
module to the case of \oneh-handles. 

Recall that in \cite{MN}, the skein lasagna module of a \twoh-handlebody was shown
to be isomorphic to the so-called {\em cabled Khovanov-Rozansky} of the
attaching link for the \twoh-handles; this is obtained from the Khovanov-Rozansky
homologies of the cables of this attaching link, modulo certain cobordism
relations. We define an analogue of the cabled Khovanov-Rozansky homology for
two links $K, L$ in the boundary of $\Wone=\natural^m(S^1 \times B^3)$ (and, more
generally, any other four-manifold); we call this the {\em cabled skein lasagna
module} $\cS(\Wone; K, L)$. 

\begin{theorem}
\label{Thm:main}
Consider \fourm-manifolds
$ \Wone \subseteq \Wtwo \subseteq \Wthree \subseteq \Wfour$ where \begin{itemize}
\item $\Wone=\natural^m(S^1 \times B^3)$ is the union of $m$ one-handles;
\item  $\Wtwo$ is obtained from $\Wone$ by attaching $n$ two-handles along a framed link
$K$;
\item $\Wthree$ is obtained from $\Wtwo$ by attaching $p$ three-handles along spheres
$S_1, \dots S_p$;
\item $\Wfour$ be obtained from $\Wthree$ by attaching some four-handles.
\end{itemize}
Consider also a framed link $L \subset \del \Wfour$, and view $K \cup L$ as a link
in $\del \Wone$. Then, the skein lasagna module $ \Sz(\Wfour; L)$ is isomorphic to the
quotient of the cabled skein lasagna module $ \cS(\Wone; K, L)$ by coequalizing
relations coming from the \threeh-handles as in Theorem~\ref{Thm:3h}. (See
Theorem~\ref{thm:main} for a more precise statement.) 
\end{theorem}

The cabled skein lasagna module $ \cS(\Wone; K, L)$ is constructed from the
invariants $\Sz(\Wone; K(a, b) \cup L)$ where $K(a, b) \cup L$ is a family of framed
links in $\del \Wone = \#^m (S^1 \times S^2)$ consisting of $L$ and cables $K(a, b)$
of the attaching link $K$ for the \twoh-handles. Thus, Theorem~\ref{Thm:main}
allows us to express $\Sz(\Wfour; L)$ in terms of skein lasagna modules of links in
$\del \Wone$ (and maps between them).  

The second half of our paper studies in more detail the skein lasagna modules
for links in $\del \Wone$ where $\Wone=\natural^m(S^1 \times B^3)$. We work with
coefficients in a field $\k$. By cutting along the cocores of the \oneh-handles,
we reduce the problem of computing $\Sz(\Wone; L, \k)$ to a problem about skein
lasagna modules for the (boundary of the) \zeroh-handle $B^4$ with a family of
framed links related to $L$. For links in $B^4$, the invariant $\Sz$
is simply the Khovanov-Rozansky homology $\KhRN$. 

\begin{theorem}
    \label{Thm:mainonehandle}
Let $\Wone = \natural^m(S^1 \times B^3)$ with a nullhomologous link $L\subset \del
\Wone$ in the boundary that intersects the belt spheres of the \oneh-handles
transversely in $2p_i$ points for $1\leq i\leq m$. Let $R\subset S^3 \setminus
\bigsqcup_i (B_i\cup \overline{B_i})$ denote the tangle obtained from $L$ by
cutting open along the belt spheres. Then, the skein lasagna module $\Sz(\Wone; L,
\k)$ is isomorphic to the quotient 
$$
\bigoplus_{\substack{\mathrm{tangles } ~ T_i \\|\del T_i|=2p_i}} 
\KhRN(R\cup \bigsqcup_i (T_i \sqcup \overline{T_i}), \k) 
\{(\textstyle\sum_i p_i)(N-1)\}  
\big/ \sim
$$
where $\{ \cdot \}$ denotes a grading shift, and the relation $\sim$ is given by
taking coinvariants for the actions of certain categories $\Sz(B^3;P_{p_i})$
associated to the configurations $P_{p_i}$ of $p_i$ positively oriented and
$p_i$ negatively oriented points in $S^2 =\del B^3$. (See
Theorem~\ref{thm:mainonehandle} for a more precise statement.)
\end{theorem}

Furthermore, we will show that the isomorphisms from
Theorem~\ref{Thm:mainonehandle} are functorial in the following sense: They
allow an expression of maps associated to cobordisms $S\subset \del \Wone \times I$
between links $S\colon L\to L'$ in  $\del \Wone = \#^m(S^1 \times S^2)$ in terms of
components computed entirely from maps associated to link cobordisms in $S^3$. 

By combining Theorems~\ref{Thm:main} and \ref{Thm:mainonehandle} (plus the
functoriality statement), we thus obtain a recipe for expressing the lasagna
skein modules of any four-manifold in terms of Khovanov--Rozansky homologies of
links in $S^3$ and maps associated to cobordisms in $S^3 \times I$. The
invariant is a quotient of a (typically infinite) direct sum of homologies of
links by a subspace defined in terms of link cobordism maps. 

\begin{remark}
    Although the invariant $\Sz(\Wgeneric; L,\k)$ for any four-manifold $\Wgeneric$ can be expressed purely in
terms of link homology in $S^3$, specifically $\KhR_N$, it would be difficult to
prove directly that these expressions yield a four-manifold invariant. A direct
proof of invariance, without comparing to the intrinsically defined invariants
$\Sz$, would require checking handle slide and handle cancellation moves as well
as higher coherence conditions between their composites. Handle slides for
2-handles are studied (for $N=2$) in \cite{HRW} and instances of $(2,3)$-handle
cancellation are discussed in Example~\ref{ex:cancellation}. Another interesting
question concerns the behaviour of our algebraic decription of $\Sz(\Wgeneric; L,\k)$ under
reversing the handle decomposition of $\Wgeneric$. However, our approach uses
transversality arguments to isotope skeins away from cocores of handles to yield
simplified handle formulas; hence, we do not expect these formulas to reflect the duality between
$k$- and $(4-k)$-handles, because the duality does not respect cocores.
\end{remark}

Specializing the setting of Theorem~\ref{Thm:mainonehandle} to the case of a
single \oneh-handle, we consider the link $S^1 \times P_p \subset S^1 \times B^3$
consisting of $2p$ parallel circles, with $p$ of them oriented one way and $p$
the other way. We prove that $\Sz(S^1 \times B^3, S^1 \times P_p)$ is isomorphic
to the zeroth Hochschild homology of the category $\Sz(B^3;P_{p_i})$. From here
we get the following explicit calculation for $N=2$.

\begin{theorem} 
\label{Thm:Pp}
The skein lasagna module $\Sztwo(S^1 \times B^3; S^1 \times P_p, \k)$ is
\begin{enumerate}[(a)]
\item one-dimensional when $p=0$;
\item four-dimensional when $p=1$;
\item infinite dimensional when $p \geq 2$.  
\end{enumerate}
\end{theorem}

Using methods analogous to those employed in part (a), we also show that
$\Sztwo(S^1 \times S^3, \k)$ is one-dimensional; see
Corollary~\ref{cor:SoneBthree}. For part (c), we actually show that $\Sztwo(S^1
\times B^3, S^1 \times P_p, \k)$ is infinite dimensional in bidegree $(0, 0)$.
This answers in the negative Question 1.7 from \cite{MN}, about whether skein
lasagna modules are always locally finite dimensional, i.e., finite dimensional in each fixed bidegree and homology class. 

This still leaves open the following:

\begin{question} If $\Wgeneric$ is simply connected, is $\Sz(\Wgeneric;L,\k)$ always locally
finite dimensional?
\end{question}

For $\Wone=  \natural^m(S^1 \times B^3)$, one can view the skein lasagna module
$\Sztwo(\Wone; L)$ as a variant of Khovanov  homology for links $L$ in $\#^m (S^1
\times S^2)$. Another version of Khovanov homology for these links was
constructed by Rozansky (for $m=1$) in \cite{Rozansky}, and Willis \cite{Willis}
for arbitrary $m$. The Rozansky--Willis homology $\RW(L)$ is finitely generated
in each bidegree and, thus, different from our theory. We expect that $\RW(L)$
appears on the $E_2$ page of a spectral sequence converging to $\Sztwo(\Wone; L)$.
See Section~\ref{sec:RW} for a further discussion and Section~\ref{sec:spec} for
a conjectural extension of the Rozansky--Willis homology to links in the
boundary of other \fourm-manifolds.

\medskip

{\bf Organization of the paper.} In Section~\ref{sec:prelim} we go over a few
preliminaries about skein lasagna modules and Kirby diagrams. In
Section~\ref{sec:23} we study the behavior of skein lasagna modules under
attaching \twoh- and \threeh-handles, proving Theorems~\ref{Thm:3h} and
~\ref{Thm:main}. In Section~\ref{sec:1h} we focus on \oneh-handles, and prove
Theorems~\ref{Thm:mainonehandle} and ~\ref{Thm:Pp}.
\medskip

{\bf Conventions.} All the manifolds considered in this paper will be smooth,
compact, and oriented. All links and surfaces are oriented and normally framed.

\medskip

{\bf Acknowledgements.} This paper builds on previous joint work and many
enlightening conversations of KW and PW with Scott Morrison, without which this
paper probably would not exist. We would also like to thank Matthew Hogancamp
and Ikshu Neithalath for helpful comments on a draft of this paper.   
%!TEX root=skein.tex
\section{Preliminaries}
\label{sec:prelim}

\subsection{Skein lasagna modules}
We start by reviewing the construction of skein lasagna modules from \cite[Section 5.2]{MWW}. 

 Following \cite{MWW} and \cite{MN}, for a framed link $L \subset \R^3$, we
 write 
$$\KhRN(L) = \bigoplus_{i, j\in \Z} \KhRN^{i, j}(L)$$ for the $\gl_N$ version of
Khovanov-Rozansky homology. Here, $i$ denotes the homological grading and $j$
denotes the quantum grading. 

If we have an oriented manifold $S$ diffeomorphic to the standard $3$-sphere
$S^3$, and a framed link $L \subset S$, we can define a canonical invariant
$\KhRN(S,L)$ as in \cite[Definition 4.12]{MWW}. We sometimes drop $S$ from the
notation and simply write $\KhRN(L)$. 

Given a framed cobordism $\Sigma \subset S^3  \times [0,1]$ from $L_0$ to $L_1$,
there is an induced map
$$ \KhRN(\Sigma) \colon \KhRN(L_0) \to \KhRN(L_1)$$ which is homogeneous of bidegree
$(0, (1-N)\chi(\Sigma))$. 

Let $W$ be a \fourm-manifold and $L \subset \del W$ a framed link. A {\em lasagna
filling} $F=(\Sigma, \{(B_i,L_i,v_i)\})$ of $W$ with boundary $L$ consists of
\begin{itemize}
\item A finite collection of disjoint $4$-balls $B_i$ (called {\em input balls})
embedded in the interior or $W$;
\item A framed oriented surface $\Sigma$ properly embedded in $W \setminus
\cup_i B_i$, meeting $\del W$ in $L$ and meeting each $\del B_i$ in a link
$L_i$; and
\item for each $i$, a homogeneous label $v_i \in \KhRN(\del B_i, L_i).$
\end{itemize}
The bidegree of a lasagna filling $F$ is 
$$ \deg(F) := \sum_i \deg(v_i) + (0,(1-N)\chi(\Sigma)).$$ If $W$ is a $4$-ball,
we can define a cobordism map
$$ \KhRN(\Sigma)\colon \bigotimes_i \KhRN(\del B_i, L_i) \to \KhRN(\del W, L)$$ and
an evaluation
$$ \KhRN(F) := \KhRN(\Sigma)(\otimes_i v_i) \in \Kh(\del W, L).$$

We define the skein lasagna module as the bigraded abelian group
$$ \Sz(W; L) := \Z\{ \text{lasagna fillings $F$ of $W$ with boundary
$L$}\}/\sim$$ where $\sim$ is the transitive and linear closure of the following
relation: 
\begin{enumerate}[(a)]
\item Linear combinations of lasagna fillings are set to be multilinear in the
labels $v_i$;
\item Furthermore, two lasagna fillings $F_1$ and $F_2$ are set to be equivalent
if  $F_1$ has an input ball $B_i$ with label $v_i$, and $F_2$ is obtained from
$F_1$ by replacing $B_i$ with another lasagna filling $F_3$ of a $4$-ball such
that $v_i=\KhRN(F_3)$, followed by an isotopy rel $\del W$ (where the isotopy is
allowed to move the input balls):
\[
    %\label{eq:lasagnaequiv} 
    \begin{picture}(0,0)%
\includegraphics{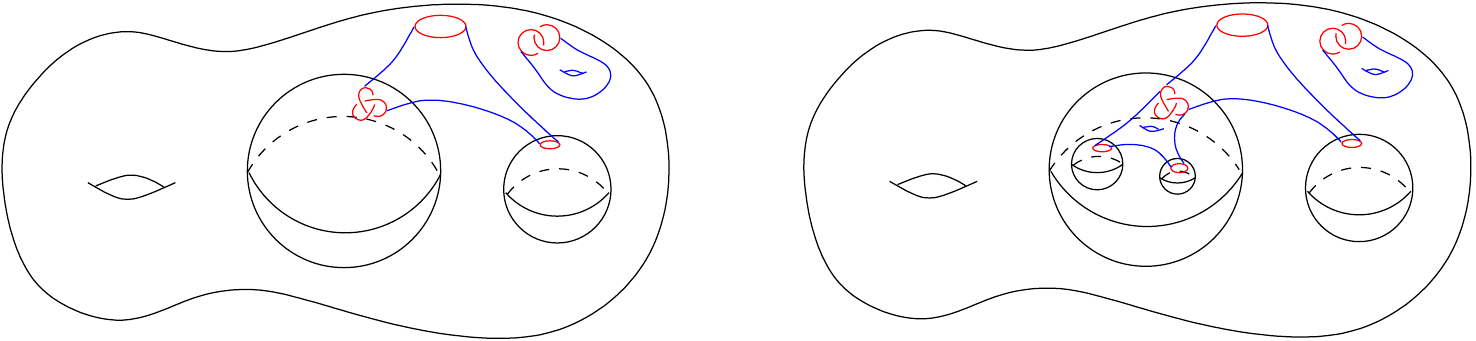}%
\end{picture}%
\setlength{\unitlength}{3158sp}%
\begingroup\makeatletter\ifx\SetFigFont\undefined%
\gdef\SetFigFont#1#2#3#4#5{%
  \reset@font\fontsize{#1}{#2pt}%
  \fontfamily{#3}\fontseries{#4}\fontshape{#5}%
  \selectfont}%
\fi\endgroup%
\begin{picture}(8837,2037)(674,-2732)
\put(7046,-1482){\makebox(0,0)[lb]{\smash{{\SetFigFont{10}{12.0}{\rmdefault}{\mddefault}{\updefault}{\color[rgb]{0,0,0}$v_k$}%
}}}}
\put(4957,-1778){\makebox(0,0)[lb]{\smash{{\SetFigFont{10}{12.0}{\rmdefault}{\mddefault}{\updefault}{\color[rgb]{0,0,0}$\sim$}%
}}}}
\put(7758,-2400){\makebox(0,0)[lb]{\smash{{\SetFigFont{10}{12.0}{\rmdefault}{\mddefault}{\updefault}{\color[rgb]{0,0,0}$B_i$}%
}}}}
\put(8878,-2320){\makebox(0,0)[lb]{\smash{{\SetFigFont{10}{12.0}{\rmdefault}{\mddefault}{\updefault}{\color[rgb]{0,0,0}$B_j$}%
}}}}
\put(7304,-2209){\makebox(0,0)[lb]{\smash{{\SetFigFont{10}{12.0}{\rmdefault}{\mddefault}{\updefault}{\color[rgb]{0,0,0}$F_3$}%
}}}}
\put(2947,-2408){\makebox(0,0)[lb]{\smash{{\SetFigFont{10}{12.0}{\rmdefault}{\mddefault}{\updefault}{\color[rgb]{0,0,0}$B_i$}%
}}}}
\put(4067,-2328){\makebox(0,0)[lb]{\smash{{\SetFigFont{10}{12.0}{\rmdefault}{\mddefault}{\updefault}{\color[rgb]{0,0,0}$B_j$}%
}}}}
\put(5898,-2390){\makebox(0,0)[lb]{\smash{{\SetFigFont{10}{12.0}{\rmdefault}{\mddefault}{\updefault}{\color[rgb]{0,0,0}$F_2$}%
}}}}
\put(1034,-2385){\makebox(0,0)[lb]{\smash{{\SetFigFont{10}{12.0}{\rmdefault}{\mddefault}{\updefault}{\color[rgb]{0,0,0}$F_1$}%
}}}}
\put(2569,-1328){\makebox(0,0)[lb]{\smash{{\SetFigFont{10}{12.0}{\rmdefault}{\mddefault}{\updefault}{\color[rgb]{0,0,0}$v_i$}%
}}}}
\put(4054,-1661){\makebox(0,0)[lb]{\smash{{\SetFigFont{10}{12.0}{\rmdefault}{\mddefault}{\updefault}{\color[rgb]{0,0,0}$v_j$}%
}}}}
\put(8858,-1661){\makebox(0,0)[lb]{\smash{{\SetFigFont{10}{12.0}{\rmdefault}{\mddefault}{\updefault}{\color[rgb]{0,0,0}$v_j$}%
}}}}
\put(7874,-1763){\makebox(0,0)[lb]{\smash{{\SetFigFont{10}{12.0}{\rmdefault}{\mddefault}{\updefault}{\color[rgb]{0,0,0}$v_l$}%
}}}}
\end{picture}%

\]
\end{enumerate}

For future reference, here is a useful lemma.
\begin{lemma}
\label{lem:fixballs} 
Let $W$ and $L$ be as above, and fix balls $R_1, \dots, R_n$, one in each
connected component of $W$. Then, the equivalence relation defining $\Sz(W; L)$
can be alternatively be described as the transitive and linear closure of the
following relation: 
\begin{itemize}
\item Linear combinations of lasagna fillings are set to be multilinear in the
labels $v_i$;
\item Lasagna fillings that are isotopic rel $\del W$ are set to be equivalent;
\item Two lasagna fillings are also set to be equivalent if they differ as in
(b) above, where the input ball $B_i$ is one of the chosen balls $R_1, \dots,
R_n$.
\end{itemize}
\end{lemma}

\begin{proof}
If $F_1$ and $F_2$ are equivalent as in the lemma, let us show that they are
equivalent as in the definition of the skein lasagna module. The only new
relation is the isotopy, which can be thought of as a particular instance of
(b), where $B_1$ is replaced by a slightly smaller ball with the same decoration
(and $F_3$ is a product cobordism). 

Conversely, if $F_1$ and $F_2$ are equivalent as in the definition of the skein
lasagna module, we only have to consider the case when they are related by (b).
We can then isotope $B_i$ to turn it into the ball $R_j$ in the same connected
component, and view (b) as a combination of the moves in the lemma.
\end{proof}

Skein lasagna modules decompose according to relative homology classes, as noted
in \cite[Section 2.3]{MN}:
\begin{align}
\label{eq:decompose0}
\Sz(W;L)=\bigoplus\limits_{\alpha\in H_2(W,L; \Z)} \Sz(W;L,\alpha).
\end{align}

Observe that in the case where $L$ is not null-homologous in $W$ (i.e. $[L] \neq
0 \in H_1(W; \Z)$), then there are no lasagna fillings, so $\Sz(W; L)=0$. When
$[L]=0\in H_1(W; \Z)$, consider the boundary map in the long exact sequence of
the pair $(W, L)$:
$$ \del: H_2(W, L; \Z) \to H_1(L; \Z).$$ The only classes $\alpha\in H_2(W,L;
\Z)$ that can contribute non-trivially are those that map to the fundamental
class $[ L] \in H_1(L; \Z)$ under $\del$.  Let us introduce the notation
$$ H_2^L(W; \Z) := \del^{-1}([L]) \subseteq H_2(W,L; \Z).$$ Note that, using the
long exact sequence of the pair, the difference of two classes in $H_2^L(W; \Z)$
can be identified with an element of $H_2(W; \Z)$. Thus, $H_2^L(W; \Z)$ is a
torsor over $H_2(W; \Z)$; it can be identified with the latter group after
choosing a base element in $H_2^L(W; \Z)$.

The decomposition ~\eqref{eq:decompose0} becomes
\begin{align}
\label{eq:decompose}
\Sz(W;L)=\bigoplus\limits_{\alpha\in H_2^L(W; \Z)} \Sz(W;L,\alpha).
\end{align}
We will use the decomposition~\eqref{eq:decompose} in the case of a general link
$L$; when $[L]\ne0$, we have $H_2^L(W; \Z)=\emptyset$ and $\Sz(W; L)=0$.

\subsection{Gluing and cobordisms} \label{sec:gluing} Let us consider two
\fourm-manifolds $W$ and $Z$ that have some part $Y$ of their boundaries in common,
as follows:
$$ \del W = Y \amalg Y_0, \ \ \del Z = (-Y) \amalg Y_1,$$ where $\amalg$ denotes
disjoint union. We can glue $W$ and $Z$ along $Y$ to form a new \fourm-manifold $W
\cup Z$ with boundary $Y_0 \amalg Y_1$. Suppose we are also given links $L_0
\subset Y_0$, $L_1 \subset Y_1$ and $L \subset Y$. Let $\overline{L} \subset -Y$ denote the mirror reverse of $L$. Then, we have a map
\begin{equation}
\label{eq:glue}
\Psi: \Sz(W; L \cup L_0) \otimes \Sz(Z; \overline{L} \cup L_1) \to \Sz(W \cup Z; L_0 \cup L_1)
 \end{equation}
obtained by gluing lasagna fillings along $L$:
$$ [F] \otimes [G] \mapsto [F \cup G].$$ It is easy to see that if two lasagna
 fillings $F_1$ and $F_2$ are equivalent in $W$, and $G_1$ and $G_2$ are
 equivalent in $Z$, then $F_1 \cup G_1$ and $F_2 \cup G_2$ are equivalent in $W
 \cup Z$, so \eqref{eq:glue} is well-defined.

Starting from here, we see that skein lasagna modules are functorial under
inclusions, in the following sense. We consider the case when $Y_0 =\emptyset$,
and we fix a  lasagna filling $G$ of $Z$ with boundary $\overline{L} \cup L_1$. We can
think of $Z$ as a cobordism from $Y=\del W$ to $Y_1$. Then, there is an induced
cobordism map
\begin{equation}
\label{eq:PsiZG}
 \Psi_{Z;G}=\Psi(\cdot \otimes [G]) : \Sz(W; L) \to \Sz(W \cup Z; L_1).
 \end{equation}

Observe that the maps \eqref{eq:PsiZG} behave well with respect to compositions: 
\begin{equation}
\label{eq:compose}
 \Psi_{Z'; G'} \circ \Psi_{Z; G} = \Psi_{Z \cup Z'; G \cup G'}.
  \end{equation}
  
Furthermore, in terms of the decompositions \eqref{eq:decompose}, given $\alpha
\in H^L_2(W; \Z)$, by attaching to it the class of $G$ in $H_2^{L \cup L_1}(Z;
\Z)$ we get a class $\alpha_1 \in H_2^{L_1}(W \cup Z; \Z)$. Then, $\Psi_{Z; G}$
maps $\Sz(W; L, \alpha)$ to $\Sz(W \cup Z; L_1, \alpha_1)$. We let
\begin{equation}
\label{eq:PsiZGa}
 \Psi_{Z;G, \alpha} : \Sz(W; L, \alpha) \to \Sz(W \cup Z; L_1, \alpha_1)
 \end{equation}
 denote the restriction of $\Psi_{Z; G}$.

When the lasagna filling $G$ consists of a surface $S$ (an embedded cobordism $S
\subset Z$ from $L$ to $L_1$) with no input balls, we will simply write $\Psi_{Z;
S, \alpha}$ for $\Psi_{Z; G, \alpha}$. Furthermore, we could decorate $S$ with
$n$ dots at a chosen location, for $0 \leq n \leq N-1$, as usual in $\gl_N$
foams; cf. \cite[Example 2.3]{MWW}. This corresponds to constructing a lasagna
filling $S(n \bullet)$ with $n$ input balls intersecting $S$ along unknots, each
decorated with the generator $$X \in \KhRN(U) \cong \Z[X]/( X^N ).$$ (This
filling is equivalent to one where we consider a single input ball intersecting
$S$ in an unknot, decorated with $X^n$.) When the chosen location of the dot
placement is clear from the context, then we denote the corresponding map by 
\begin{equation}
\label{eq:dots} 
\Psi_{Z;S(n \bullet), \alpha}:  \Sz(W; L, \alpha) \to \Sz(W \cup Z; L_1, \alpha_1).
\end{equation}

\subsection{Kirby diagrams}
Let $W$ be a smooth, oriented, connected, compact \fourm-manifold (possibly with
boundary). By standard Morse theory,  $W$ can be decomposed into $k$-handles for
$k=0, \dots, 4$, arranged according to their index $k$. Furthermore, without
loss of generality, we can arrange so that there is a unique \zeroh-handle, and the
number of \fourh-handles is either $0$ or $1$, according to whether $W$ has empty
boundary or not.

Denote the numbers of $1$-, $2$- and \threeh-handles by $m$, $n$ and $p$,
respectively. After attaching the \oneh-handles to the \zeroh-handle we get the
handlebody $\natural^m(S^1 \times B^3)$, with boundary $\#^m (S^1 \times S^2)$.
(Here, $\natural$ denotes the boundary connected sum, and $\#$ the usual
interior connected sum.) The attaching circles for the \twoh-handles form a link
$$ K \subset \#^m (S^1 \times S^2),$$ with components $K_1, \dots, K_n$. The
link also has a framing, which specifies how the \twoh-handles are attached. Once
these are attached, the boundary of the resulting manifold must be of the form
$Y \# p(S^1 \times S^2)$. Attaching the \threeh-handles gets rid of the $p$ summands
of $S^1 \times S^2$, so the resulting boundary is some $3$-manifold $Y$. In the
case $\del W \neq \emptyset$, we stop here and we have $\del W = Y$. In the case
where $W$ is closed, we must have $Y=S^3$ and we attach the \fourh-handle (a
four-ball) to $S^3$ at the last step to eliminate the boundary.

The handle decomposition allows us to represent $W$ by a {\em Kirby diagram}.
This consists of drawing $\#^m (S^1 \times S^2)$ as $m$ pairs of spheres in
$\R^3$, where  we think of the spheres in each pair as  identified to produce a
\oneh-handle (and we also add the point at infinity to $\R^3$). We then draw a
picture of the attaching link $K$ for the \twoh-handles, where the link can go
through the \oneh-handles. The framing of $K$ can be specified by drawing parallel
copies of the components of $K$. (The components that don't go through the
\oneh-handles can be viewed as living in $S^3$; for those, an alternative way to
specify the  framing is by an integer, which is the difference between the given
framing and the Seifert framing.) To determine $W$, in principle we should also
specify the attaching spheres for the \threeh-handles. These are usually not drawn
in the Kirby diagram. In the case where $\del W =\emptyset$, this leaves no
ambiguity, because there is a unique way to fill $\#^p (S^1 \times S^2)$ by
\threeh-handles and then by a \fourh-handle.

For example, we show here a Kirby diagram of $W=\CP^2 \# \CP^2$ with one \oneh-handle
and three \twoh-handles. For the attaching curve of the \twoh-handle that goes through
the \oneh-handle, we specified the framing by drawing a parallel copy by a dashed
curve; for the other \twoh-handles, we used numbers:
\[ %% Creator: Inkscape 1.0.1 (c497b03c, 2020-09-10), www.inkscape.org
%% Accompanies image file 'kirby.pdf' (pdf, eps, ps)
\begingroup%
  \makeatletter%
  \providecommand\color[2][]{%
    \errmessage{(Inkscape) Color is used for the text in Inkscape, but the package 'color.sty' is not loaded}%
    \renewcommand\color[2][]{}%
  }%
  \providecommand\transparent[1]{%
    \errmessage{(Inkscape) Transparency is used (non-zero) for the text in Inkscape, but the package 'transparent.sty' is not loaded}%
    \renewcommand\transparent[1]{}%
  }%
  \providecommand\rotatebox[2]{#2}%
  \newcommand*\fsize{\dimexpr\f@size pt\relax}%
  \newcommand*\lineheight[1]{\fontsize{\fsize}{#1\fsize}\selectfont}%
  \ifx\svgwidth\undefined%
    \setlength{\unitlength}{208.93454147bp}%
    \ifx\svgscale\undefined%
      \relax%
    \else%
      \setlength{\unitlength}{\unitlength * \real{\svgscale}}%
    \fi%
  \else%
    \setlength{\unitlength}{\svgwidth}%
  \fi%
  \global\let\svgwidth\undefined%
  \global\let\svgscale\undefined%
  \makeatother%
  \begin{picture}(1,0.50127611)%
    \lineheight{1}%
    \setlength\tabcolsep{0pt}%
    \put(0,0){\includegraphics[width=\unitlength,page=1]{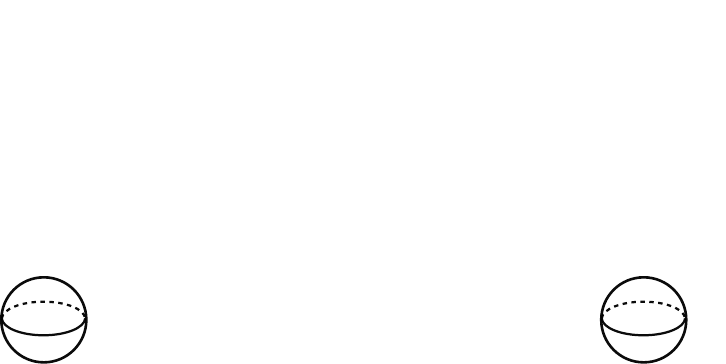}}%
    \put(0.52356608,0.46444498){\makebox(0,0)[lt]{\lineheight{1.25}\smash{\begin{tabular}[t]{l}$2$\end{tabular}}}}%
    \put(0.91194181,0.31182544){\makebox(0,0)[lt]{\lineheight{1.25}\smash{\begin{tabular}[t]{l}$1$\end{tabular}}}}%
    \put(0,0){\includegraphics[width=\unitlength,page=2]{kirbyd.pdf}}%
  \end{picture}%
\endgroup%
 \]

For more details about the subject, we refer to the book \cite{GS}.

\section{Two- and three-handles}
\label{sec:23}
\subsection{Two-handles}
\label{sec:12}
 The paper \cite{MN} contains a description of the skein lasagna module for
 \twoh-handlebodies (\fourm-manifolds $W$ made of a \zeroh-handle and some \twoh-handles),
 where the link $L \subset \del W$ is empty, or at least local (contained in a
 $3$-ball). The description is in terms of the Khovanov-Rozansky homology of
 cables of the attaching link $K$.
 
In this subsection we extend that description to the case where we attach
\twoh-handles to any \fourm-manifold $\Wgeneric$, to obtain a new manifold $\Whandle$.
Moreover, we do not impose any restriction on the link $L \subset \del \Wgeneric$. The
formula is very similar to that in \cite{MN}. The role of the Khovanov-Rozansky
homology $\KhR_N$ will be played by the skein lasagna module $\Sz(\Wgeneric; -)$, which can be
thought of as a link homology for links in the boundary of  $\Wgeneric$. (When $\Wgeneric=B^4$,
we have $\Sz(\Wgeneric; L) = \KhR_N(L)$.) 

Let $K_1, \dots, K_n$ be the components of the framed link $K \subset \del \Wgeneric$
along which the \twoh-handles are attached. The framing gives diffeomorphisms
$f_i$ between tubular neighborhoods $\nu(K_i)$ of each $K_i$ and $S^1 \times
D^2$. Given $n$-tuples of nonnegative integers
$$k^-=(k_1^-,\dots,k_n^-),\ \ \ k^+=(k_1^+,\dots,k_n^+),$$ we let $K(k^-,k^+)$
denote the framed, oriented cable of $K$ consisting of $k_i^-$ negatively
oriented parallel strands to $K_i$ and $k_i^+$ positively oriented parallel
strands. Here, the notion of parallelism for the strands is determined by the
framing, that is,
$$ K(k^-,k^+) = \bigcup_i f_i^{-1}(S^1 \times \{x_1^-, \dots, x_{k_i^-}^-,
x_1^+, \dots, x_{k_i^+}^+\})$$ for fixed points $ x_1^-, \dots, x_{k_i^-}^-,
x_1^+, \dots, x_{k_i^+}^+ \in D^2.$

After attaching \twoh-handles to $\Wgeneric$ along $K$, we obtain the manifold $\Whandle$. Suppose we
are given a framed link $L \subset \del \Whandle$. Generically, we can assume that $L$
stays away from the attaching regions of the \twoh-handles, and therefore we can represent it as
a link in $\del \Wgeneric$, disjoint from (but possibly linked with) $K$. (There are various ways of isotoping $L$ off of the attaching regions; the results of the calculation will be isomorphic.) We let
$$K(k^-, k^+) \cup L$$ be the union of $K(k^-,k^+)$ and $L$, where we do the
cabling on the components of $K$ by choosing the tubular neighborhoods of $K_i$
to be disjoint from $L$. (Note that $K(k^-, k^+) \cup L$ is not a split disjoint
union.)

We seek to express the skein lasagna module $\Sz(\Whandle; L)$ in terms of $\Sz(\Wgeneric;
K(k^-, k^+) \cup L)$. To do this, we need to introduce a few more notions.

For each $i$, let $B_{k_i^-, k_i^+}$ be the subgroup of the braid group on
$k_i^- + k_i^+$ strands that consists of self-diffeomorphisms of $D^2$ rel
boundary (modulo isotopy rel boundary) taking the set $\{x_1^-, \dots,
x_{k_i^-}^-\}$ to itself and the set $\{x_1^+, \dots, x_{k_i^+}^+\}$ to itself.
By taking the product with the identity on $S^1$, a braid element $b \in
B_{k_i^-, k_i^+}$ induces a self-diffeomorphism of $D^2 \times S^1$, which can
be pulled back (via $f_i$) to a self-diffeomorphism of $\nu(K_i)$. This gives a
group action
$$\beta_i : B_{k_i^-, k_i^+} \to \Aut(\Sz(\Wgeneric; K(k^-, k^+) \cup L)).$$

Let $e_i\in\Z^n$ denote the $i^{th}$ basis vector.  
Two strands parallel to $K_i$, if they have opposite orientations, co-bound a
ribbon band $R_i$ in $S^3$. By pushing $R_i$ into $S^3 \times [0,1]$ so that it
is properly embedded there, and taking the disjoint union with the identity
cobordisms on the other strands, we obtain an oriented cobordism (still denoted
$R_i$) from $K(k^-,k^+) \cup L$ to $K(k^-+e_i, k^++e_i)\cup L$. For $d=0, 1,
\dots, N-1$, we can decorate  $R_i$  with $d$ dots, and obtain a cobordism map
$$ \psi^{[d]}_i : \Sz(\Wgeneric;K(k^-,k^+) \cup L) \to \Sz(\Wgeneric;K(k^-+e_i, k^++e_i) \cup
L),$$ which changes the bigrading by $(0, 2d)$. 

Next, recall that we have a decomposition \eqref{eq:decompose} for the skein
lasagna module $\Sz(\Whandle; L)$, according to homology classes in $H_2^L(\Whandle; \Z)$. Let
us see how these homology classes are related to the similar ones in $\Wgeneric$.
Consider the tubular neighborhood $\nu(K) =\cup_i \nu(K_i)$, which is a union of
solid tori. Express $\Whandle$ as the union
$$ \Whandle = \Wgeneric \cup C \cup Z,$$ where $Z$ is the union of the new \twoh-handles, and
$C \cong \nu(K) \times [0,1]$ is a connecting cylinder between $\Wgeneric$ and $Z$. Let
also $$C' = \nu(K) \times \{0, 1\} \subset C.$$
We identify $\nu(K)$ with $\nu(K) \times \{0\}$ and denote $\nu(K) \times \{1\}$ by $\del_- Z$ (part of the boundary $\del Z$).

The Mayer-Vietoris sequence for $\Whandle$ relative to the union of $\Wgeneric$ and $Z \cup L$
reads
$$
 \cdots \to H_*(\Whandle, \Wgeneric\cap(Z \cup L); \Z) \to H_*(\Whandle,\Wgeneric; \Z) \oplus H_*(\Whandle, Z\cup L; \Z) \to H_*(\Whandle, \Wgeneric\cup (Z \cup L); \Z) \to \cdots
$$

Observe that, by excision, $H_3(\Whandle, \Wgeneric\cup (Z \cup L); \Z) \cong H_3(C, C'; \Z)=0$. From here we
obtain an exact sequence
\begin{equation}
\label{eq:MVi}
 0 \to H_2(\Whandle, L; \Z)\to H_2(Z, \del_-Z; \Z) \oplus H_2(\Wgeneric, \nu(K) \cup L; \Z)
\to H_2(C, C'; \Z).
\end{equation}
 Thus, an element in $H_2^L(\Whandle; \Z) \subseteq H_2(\Whandle, L; \Z)$
can be identified with its image in $ H_2(Z, \del_-Z; \Z) \oplus H_2(\Wgeneric, \nu(K)
\cup L; \Z)$, which we write as a pair $(\alpha, \eta)$. 

Let us further identify $H_2(Z, \del_-Z; \Z)$ with $\Z^n$ by letting the $i$th
handle correspond to the coordinate vector $e_i$. Then, we write 
$$ \alpha = (\alpha_1, \dots, \alpha_n) \in \Z^n$$ and let $\alpha^+$ denote its
positive part and $\alpha^-$ its negative part; i.e.,
$\alpha^+_i=\operatorname{max}(\alpha_i,0)$ and
$\alpha^-_i=\operatorname{min}(\alpha_i,0)$. We also let $|\alpha| = \sum_i
|\alpha_i|$. 

Let $r \in \N^n$ and consider the cable $K(r-\alpha^-,r+\alpha^+)$. The fact
that $(\alpha, \eta) \in \Z^n \oplus H_2(\Wgeneric, \nu(K) \cup L; \Z) $ is in the
kernel of the map to $H_2(C, C'; \Z)\cong \Z^n$ in \eqref{eq:MVi} implies the
existence of a (unique) class
$$ \eta^r \in H_2^{L \cup K(r-\alpha^-, r+\alpha^+)}(\Wgeneric; \Z)\subseteq H_2(\Wgeneric, L
\cup K(r-\alpha^-, r+\alpha^+); \Z)$$ which is sent to $\eta$ by the natural map
to $H_2(\Wgeneric, L\cup \nu(K); \Z).$

From now on, using the deformation retraction from $\nu(K)$ to $K$, let us think
of $\eta$ as a class in $H_2(\Wgeneric, K \cup L; \Z)$.

\begin{definition}
\label{def:cabled}
The cabled skein lasagna module of $K \subset \del \Wgeneric$ at level $\alpha$ and in
class $\eta$ is 
\[
\cSz(\Wgeneric; K, L, \eta) =  \Bigl( \bigoplus\limits_{r\in\N^n} \Sz( \Wgeneric; K(r-\alpha^-,r+\alpha^+) \cup L, \eta^r)\{(1-N)(2|r|+|\alpha|)\} \Bigr)/ \sim
\]
where the equivalence $\sim$ is the transitive and linear closure of the
relations
\begin{equation}
\label{eq:sim}
\beta_i(b)v \sim v, \ \ \psi^{[d]}_i(v) \sim 0 \text{ for } d < N-1,\ \ \psi^{[N-1]}_i(v) \sim v
\end{equation}
for all $i=1, \dots, n$; $b \in B_{k_i^-, k_i^+}$, and $v \in  \Sz(\Wgeneric;
K(r-\alpha^-,r+\alpha^+)\cup L, \eta^r).$
\end{definition}

\begin{theorem}
\label{thm:12h}
Let $\Wgeneric$ be a four-manifold and $L \subset \del \Wgeneric$ be a framed link. Let $\Whandle$ be
obtained from $\Wgeneric$ by attaching \twoh-handles along a framed link $K$ disjoint
from $L$. Then, for each $(\alpha, \eta) \in H_2^L(\Whandle; \Z)$, we have an
isomorphism
$$ \Phi:  \cSz(\Wgeneric;K, L, \eta) \xrightarrow{\phantom{a}\cong\phantom{a}}
\Sz(\Whandle;L,(\alpha, \eta)).$$
\end{theorem}

\begin{proof}
An element $v\in \Sz( \Wgeneric; K(r-\alpha^-,r+\alpha^+) \cup L, \eta^r)$ is represented by a
linear combination of lasagna fillings $(\Sigma,\{(B_i, L_i, v_i\})$ in $\Wgeneric$,
where $\del \Sigma = K(r-\alpha^-,r+\alpha^+) \cup L \cup (\cup_i L_i)$. We define $\Phi(v)$ to be
the class of the linear combination of lasagna fillings with the same input data
$\{(B_i, L_i, v_i\}$ as $v$, but with the surfaces given by attaching to each
$\Sigma$ (along its boundary) the disjoint union of $r_i-\alpha_i^-$ negatively
oriented discs parallel to the core of $i^{th}$ \twoh-handle and $r_i+\alpha_i^+$
positively oriented such discs (union over all $i$).

We also define a map $\Phi^{-1}$ in the opposite direction, as follows. Let $F$ be a lasagna filling in $\Whandle$ with surface $\Sigma$. We isotope the input balls of $F$ to be inside $\Wgeneric$, and isotope the surface $\Sigma$ such that its intersection with the \twoh-handles consists of several disks parallel to their cores. Removing these disks produces a lasagna filling of $\Wgeneric$ with boundary on a link of the form $K(r-\alpha^-,r+\alpha^+) \cup L$. We let this be $\Phi^{-1}(F)$.

The proofs that $\Phi$ and $\Phi^{-1}$ are well-defined and inverse to each other are similar to the
proof of Theorem 1.1 in \cite{MN}, which dealt with the case $\Wgeneric=B^4$ and
$L=\emptyset$. The extension to arbitrary $\Wgeneric$ and $L$ is obtained by replacing the Khovanov-Rozansky homologies $\KhR_N$ with the skein lasagna modules in $\Wgeneric$. (In the formulation here, the proof
of the statement is even slightly clearer since it relates lasagna skein modules
with lasagna skein modules. In particular, we do not have to choose standard
lasagna fillings with ``slighly smaller input balls'', as these were only
required when comparing $\Sz(B^4,-)$ with $\KhR_N$.)
\end{proof}

\begin{remark}
    In some cases it is known that the braid group actions on the link homology
    of cabled links factor through the symmetric group. For Khovanov homology of
    links in $\mathbb{R}^3$, this was shown by
    Grigsby--Licata--Wehrli~\cite[Theorem 2]{MR3731256}. For the $\glN$ homology
    of links in $\mathbb{R}^3$ (or $S^3$) a similar argument works in the case
    of \emph{parallelly oriented} strands \cite[Section
    6.1]{2019arXiv190404481G}. We have no reason to doubt that the same could be
    true for anti-parallel strands, i.e. in the situation relevant for $\Sz$,
    but we do not currently know how to prove it. 
\end{remark}

We will primarily be using the results from this subsection in the case where the role of $\Wgeneric$ is played by 
$$ \Wone: = \natural^m(S^1 \times B^3),$$ a manifold obtained from a \zeroh-handle by
attaching some \oneh-handles. We denote $\Whandle$ by $\Wtwo$. Then, $H_2(\Wone; \Z) = 0$, so $H_2^L(\Wone; \Z)=0$ for any null-homologous  $L$, and the decomposition \eqref{eq:decompose} for skein lasagna
modules of links in $\Wone$ is trivial (consists of a single summand). Moreover, in this case an element
$(\alpha, \eta) \in H_2^L(\Wtwo; \Z) \subseteq H_2(\Wtwo, L; \Z)$ is uniquely determined
by its image $\alpha$ in $H_2(\Wtwo, \Wone; \Z)\cong \Z^n$. Indeed, the exact sequence 
$$ 0=H_2(\Wone; \Z) \to H_2(\Wone, L \cup \nu(K); \Z) \to H_1(L \cup \nu(K); \Z)$$
show that the component $\eta$ is determined by its image in 
$$H_1(L \cup \nu(K); \Z)= H_1(L; \Z) \oplus H_1(\nu(K); \Z).$$
The part in $H_1(L; \Z)$ has to be the fundamental class $[L]$, while the part in $H_1(\nu(K); \Z)\cong \Z^n$ is the image of $\alpha$ under the isomorphisms 
$$H_2(\Wtwo, \Wone; \Z)  \xrightarrow{\cong} H_2(Z, \del_- Z; \Z) \xrightarrow{\cong} H_1(\del_-Z; \Z)  \xrightarrow{\cong} H_1(\nu(K); \Z).$$

Therefore, in this case the
class $\eta$ is redundant (being determined by $\alpha$), so we simply drop it
from the notation, writing for example $\alpha$ instead of $(\alpha, \eta)$ for
the classes in $ H_2^L(\Wtwo; \Z)$. With this in mind, the isomorphism from
Theorem~\ref{thm:12h} is written as
\begin{equation}
\label{eq:newphi} \Phi:  \cSz(\Wone;K, L) \xrightarrow{\phantom{a}\cong\phantom{a}}
\Sz(\Wtwo;L, \alpha).
\end{equation}

\subsection{Three-handles}

In \cite[Proposition 2.1]{MN} the following result was shown:
\begin{proposition}
    \label{prop:threefour}
Let $i\colon W \to W'$ be the inclusion of a \fourm-manifold $W$ into $W'$. Then
we have a natural map
\[i_*\colon \Sz(W;\emptyset) \to \Sz(W',\emptyset).\] If $W'$ is the result of a
$k$-handle attachment to $W$, then $i_*$ is a surjection for $k=3$ and an
isomorphism for $k=4$.
\end{proposition}

\begin{corollary}
    \label{cor:Sfour}
    We have $\Sz(S^4)\cong \Z$, concentrated in bidegree zero. 
\end{corollary}

In this section we focus on the case of \threeh-handle attachments. We will
generalize the statement of Proposition~\ref{prop:threefour} to \threeh-handle
attachments in the presence of boundary links and explicitly describe the kernel
of the resulting maps on $\Sz$.

Consider the following setting. Let $W$ be a \fourm-manifold with a framed link $L
\subset Y=\del W$ and an embedded $2$-dimensional sphere $S \subset Y$, disjoint
from $L$. Let $Z$ be the cobordism given by attaching a \threeh-handle to $W$ along
$S$, and let 
$$ W' = W \cup Z.$$ Let $Y'=\del W'$ be the outgoing boundary of $Z$, so that
$\del Z = (-Y) \cup Y'$. Inside $Z$ we have the two-dimensional annular
cobordism $A=I \times L$, from $L=\{0\} \times L$ to a new link $L' = \{1\}
\times L$. Given $\alpha' \in H_2^L(W'; \Z) \cong H_2^L(W; \Z) / ([S])$, let us consider the set of all $\alpha \in H_2^L(W; \Z)$ whose equivalence class modulo $[S]$ is $\alpha'$:
$$\ap := \{ \alpha  \in H_2^L(W; \Z) \mid \alpha \text{ mod } [S] = \alpha' \}.$$
 
We obtain a cobordism map as in
\eqref{eq:PsiZGa}:
$$ \Psi_{Z;A, \alpha} : \Sz(W; L, \alpha) \to \Sz(W'; L', \alpha').$$
Let
$$ \Psi_{Z;A, \alpha'} :=  \sum_{\alpha \in \ap} \Psi_{Z;A, \alpha}: \bigoplus_{\alpha \in \ap} \Sz(W; L, \alpha) \to \Sz(W'; L', \alpha').$$ 

\begin{remark}
When $L =\emptyset$ (and therefore $A=\emptyset$), then $ \Psi_{Z;\emptyset}$ is
exactly the map $i_*$ from Proposition~\ref{prop:threefour}.
\end{remark}

Let $J$ be the equator of $S$ (which is an unknot in $Y$). Equip $J$ with an
arbitrary orientation. By pushing a hemisphere of $S$ slightly from $Y =  \{0\}
\times  Y$ into the cylinder $I \times Y$, and taking its union with $I \times L
$, we obtain a  properly embedded cobordism in $I \times Y$, going from $L\cup
J$ to $L$. There are two such hemispheres, which produce two cobordisms, denoted
$\Delta_+$ and $\Delta_- \subset I \times Y$. We orient $\Delta_+$ and
$\Delta_-$ so that their boundary orientation is the one on $J$. (Note that they
are therefore ``oppositely oriented,'' in the sense that they do not match up to
produce an orientation on $S$.) Let us identify $W \cup (I \times Y)$ with $W$
itself using a standard collar neighborhood. Then, the cobordism maps associated
to $\Delta_+$ and $\Delta_-$ take the form
$$ \Psi_{I \times Y ; \Delta_+, \alpha}\colon \Sz(W; L\cup J, \alpha + [\Delta_+])\to \Sz(W; L, \alpha),$$
$$ \Psi_{I \times Y; \Delta_-, \alpha}\colon \Sz(W; L\cup J, \alpha + [\Delta_-])\to \Sz(W; L, \alpha).$$  
From here we get direct sum maps
$$\Psi_{I \times Y ; \Delta_+, \alpha'} := \bigoplus_{\alpha \in \ap} \Psi_{I \times Y ; \Delta_{+}, \alpha}$$
and
$$\Psi_{I \times Y ; \Delta_-, \alpha'} := \bigoplus_{\alpha \in \ap} \Psi_{I \times Y ; \Delta_{-}, \alpha}.$$
Observe that these two maps have the same domain
$$\bigoplus_{\alpha \in \ap}\Sz(W; L\cup J, \alpha + [\Delta_+]) = \bigoplus_{\alpha \in \ap} \Sz(W; L\cup J, \alpha + [\Delta_-])$$
and the same range $\bigoplus_{\alpha \in \ap}  \Sz(W; L, \alpha).$ Let 
$$f: = \Psi_{I \times Y ; \Delta_+, \alpha'}- \Psi_{I \times Y; \Delta_-,
\alpha'}.$$ 

\begin{theorem}
\label{thm:3h}
The map $\Psi_{Z; A, \alpha'}$ associated to a \threeh-handle addition from $W$ to $W'$
is surjective, and its kernel is exactly the image of $f$. Therefore, $\Sz(W',
L', \alpha')$ is isomorphic to $$\Bigl( \bigoplus_{\alpha \in \ap} \Sz(W, L, \alpha)\Bigr)/ \im(f),$$ that is, to the
coequalizer of the maps $\Psi_{I \times Y ; \Delta_+, \alpha'}$ and $\Psi_{I
\times Y; \Delta_-, \alpha'}$.
\end{theorem}

\begin{proof}
We first show that $\Psi_{Z; A, \alpha'}$ vanishes on the image of $f$, that is,
$$ \Psi_{Z; A, \alpha'} \circ \Psi_{I \times Y ; \Delta_+, \alpha'} =  \Psi_{Z; A,
\alpha'} \circ \Psi_{I \times Y ; \Delta_-, \alpha'}.$$ Indeed, from the
composition law \eqref{eq:compose} we see that the left hand side is associated
to the surface cobordism $\Delta_+ \cup A$ and the right hand side to $\Delta_-
\cup A$. However, inside the \threeh-handle $Z$, the sphere $S$ gets filled with a
core $B^3$, and therefore $\Delta_+$ and $\Delta_-$ are isotopic rel boundary.
It follows that the two cobordism maps are the same.

Therefore, $\Psi_{Z; A, \alpha}$ factors through a map
$$ \Phi\colon \Bigl( \bigoplus_{\alpha \in \ap} \Sz(W, L, \alpha)\Bigr)/ \im(f) \to \Sz(W', L', \alpha').$$ We need to prove
that $\Phi$ is bijective. For this, we construct its inverse $\Phi^{-1}$. Given
a lasagna filling $F'$ of $W'$ with boundary $L'$, observe that the cocore of
the \threeh-handle $Z$ is one-dimensional, and therefore we can isotope $F'$ to be
disjoint from this cocore; after this, we can push it into $W$, to obtain a
lasagna filling there, called $F$, with boundary $L$. We set 
$$\Phi^{-1}[F']=[F].$$ To see that $\Phi^{-1}$ is well-defined, we need to check
that if two lasagna fillings $F'_0$ and $F'_1$ are equivalent in $W$, then the
corresponding fillings $F_0$ and $F_1$ differ (up to equivalences in $W$) by an
element of $\im(f)$. We use Lemma~\ref{lem:fixballs}, in which we fix balls $R_i
\subset W$ away from the \threeh-handle, and consider the equivalences listed in
the lemma (with the ball replacements happening in $R_i$). Then, the
equivalences in $W'$ give rise to equivalences in $W$, with one exception: an
isotopy of the surfaces may intersect the one-dimensional cocore of $Z$ (which
is an interval). Generically, this happens in a finite set of points, each point
at a different time during the isotopy. Every time the isotopy meets the cocore,
the corresponding surfaces in $W$ differ by replacing a hemisphere of $S$ (with
boundary some closed curve $\gamma$) with its complement in $S$. Up to an
isotopy supported near $S$, we can assume that $\gamma$ is the equator $J$ with
its chosen orientation. (For example, if $\gamma$ is $J$ with the opposite
orientation, we can rotate it by $\pi$ about a transverse axis to get $J$ with
the original orientation.) Then, the hemispheres being interchanged are
$\Delta_+$ and $\Delta_-$ and hence the classes of $F_0$ and $F_1$ differ by an
element in the image of $f$.

This shows that $\Phi^{-1}$ is well-defined, and its definition makes it clear
that it is an inverse to $\Phi$. It follows that $\Phi$ is bijective, and the
conclusions follow.
\end{proof}

\begin{example}
\label{ex:cancellation}
Let $W = S^2 \times D^2$ and $S$ the sphere $S^2 \times \{p\}$, where $p\in \del
D^2$. Then attaching the \threeh-handle gives $W' = B^4$. Let us see what
Theorem~\ref{thm:3h} gives in this case. For simplicity, we ignore the
decomposition into relative homology classes. 

The skein lasagna module of $W$ has the structure of a commutative algebra over
$\Z$, with the multiplication given by putting lasagna fillings side-by-side, in
the decomposition
$$ (S^2 \times D^2) \cup_{S^2 \times I} (S^2 \times D^2) \cong S^2 \times D^2,$$
where $I \subset \del D^2$ is an interval. As a $\Z$-algebra, $\Sz(W;
\emptyset)$ was computed in \cite[Theorem 1.2]{MN} to be
$$ \Sz(W; \emptyset) \cong \Z[A_1, \dots, A_{N-1}, A_0, A_0^{-1}]$$ where $A_i$
comes from the lasagna filling corresponding to the closed surface $S^2 \times
\{0\}$, equipped with the standard orientation, and marked with $N-1-i$ dots.
(As mentioned in Section~\ref{sec:gluing}, this is equivalent to introducing one
input ball intersecting $S^2 \times \{0\}$ in an unknot labeled $X^{N-1-i}$.) 

The cobordism maps
$$ \Psi_{I \times Y; \Delta_+}, \  \Psi_{I \times Y; \Delta_-}: \Sz(W; J) \to
\Sz(W; \emptyset)$$ are as follows. The unknot $J$ is contained in a ball in the
boundary of $W$ (say, a neighborhood of the disk $\Delta_+$). Then, according to
\cite[Corollary 1.5]{MN}, we have
$$ \Sz(W; J) \cong \Sz(W) \otimes_{\Z} \KhRN(J) \cong  \Sz(W) \otimes_{\Z}
\bigl( \Z[X]/(X^N)\bigr).$$ (Strictly speaking, Corollary 1.5 in \cite{MN} is
phrased for coefficients in a field $\k$, due to the fact that its proof
requires choosing a basis of $\KhRN(J)$. In our case, $J$ is the unknot, so
$\KhRN(J)$ is free over $\Z$, and therefore the same argument applies with
coefficients in $\Z$.)

Both maps $\Psi_{I \times Y; \Delta_+}$ and $\Psi_{I \times Y; \Delta_-}$
correspond to capping the unknot by disks. The first map acts only
on the factor $\KhRN(J)$ and is given by
$$\Psi_{I \times Y; \Delta_+}(v \otimes X^{N-1-i}) = \begin{cases} v & \text{if
} i=0,\\
0 & \text{if } i =1, \dots, N-1. \end{cases}$$ A useful picture to have in mind
is that we can represent $X^{N-1-i}$ by a dotted disk (with the number of dots
specified by the exponent of $X$), which is completed by $\Delta_+$ to a dotted
sphere that bounds a ball in $W$, and hence can be evaluated to a scalar as
shown above. To compute the action of $\Psi_{I \times Y; \Delta_-}$, on the
other hand, note that the disk $\Delta_-$ completes the dotted disk to a
homologically essential dotted sphere, corresponding to a generator in $\Sz(W; \emptyset)$:
$$\Psi_{I \times Y; \Delta_-}(v \otimes X^{N-1-i}) = v\cdot A_i.$$ Therefore,
taking the coequalizer of the two maps as in Theorem~\ref{thm:3h} boils down to
setting 
$$ A_0=1, \ \ A_1=\dots = A_{N-1} =0$$ in $\Sz(W; \emptyset)$. We deduce that
$$\Sz(W'; \emptyset) \cong \Z[A_1, \dots, A_{N-1}, A_0, A_0^{-1}]/(A_1, \dots,
A_{n-1}, A_0-1) \cong \Z,$$ which is the known answer for the skein lasagna
module of $B^4$; see \cite[Example 4.6]{MWW}.
\end{example}

\begin{remark}
Example~\ref{ex:cancellation} gives an alternate formula for \threeh-handle
attachments. Let us go back to the general setting in this section, with a
\threeh-handle attached to an arbitrary \fourm-manifold $W$ along a sphere $S$ to
produce $W'$, and a framed link $L \subseteq \del W$ away from $S$. Observe that
$\Sz(W, L)$ is naturally a module over the algebra $\Sz(S^2 \times D^2;
\emptyset)$, with the module action being given by attaching fillings in a
neighborhood of the sphere $S$. It follows from the definitions that
$$ \Sz(W'; L') \cong \Sz(W; L) \otimes_{\Sz(S^2 \times D^2; \emptyset)}
\Sz(B^3 \times I; \emptyset).$$ Here, the algebra $\Sz(S^2 \times D^2;
\emptyset)$ is the free polynomial ring in $A_1, \dots, A_{N-1}, A_0, A_0^{-1}$
and $ \Sz(B^3 \times I; \emptyset)=\Sz(B^4)$ is $\Z$ as a module over that
algebra, where $A_0$ acts by $1$ and the other $A_i$ by $0$. We conclude that
$$ \Sz(W'; L') \cong \Sz(W; L)/(A_0-1, A_1, \dots, A_N).$$ 
\end{remark}

\subsection{Handle decompositions}
Let us now specialize the addition of \threeh-handles to the case where the initial
manifold $\Wgeneric=\Wtwo$ is a union of $0$-, $1$- and \twoh-handles. We will then have available to
us the description of $\Sz(\Wtwo; L, \alpha)$ from Section~\ref{sec:12}.

 If we attach a \threeh-handle to $\Wtwo$, in terms of Kirby calculus, the attaching
 sphere $S$ can be represented as a surface $\Sigma$ (of genus $0$, and disjoint
 from $L$) with boundary some copies of the $K_i$'s (the attaching circles for
 \twoh-handles). Then $S$ is the union of $\Sigma$ and (parallel copies of) cores
 of the \twoh-handles.
 
 We draw $J \subset S$ as a small unknot away from all $K_i$, and let $\Delta_+$
 be the small disk it bounds. The other hemisphere $\Delta_-$ is the complement
 of $\Delta_+$ in $S$, and goes over some of the handles. We let 
 $$\Sigma_- = \Sigma \setminus \Delta_+ \subseteq \Delta_-.$$ This is a surface
 on $\del \Wone$ whose boundary is the union of $J$ and several copies of the
 $K_i$'s. Let $s^-_i$ be the number of copies of $K_i$ in $\del \Sigma_-$ that
 appear with the negative orientation, and $s^+_i$ the number of those with the
 positive orientation. We form the vectors
  $$s^-=(s^-_1, \dots, s^-_n), \ \ \ s^+=(s^+_1, \dots, s^+_n).$$

We proceed to describe the maps $\Psi_{I \times Y ; \Delta_+, \alpha'}$ and
$\Psi_{I \times Y; \Delta_-, \alpha'}$ in this case. By Theorem~\ref{thm:12h}
with notation as in \eqref{eq:newphi}, the range $\bigoplus_{\alpha \in \ap}  \Sz(\Wtwo; Z, \alpha)$ of these
maps is identified with the direct sum of cabled skein lasagna modules $\bigoplus_{\alpha \in \ap}  \cSz(\Wone;K, L)$.
Similarly, their domain is identified with
\begin{align*}
 \bigoplus_{\alpha \in \ap}  \cSz(\Wone; K, L\cup J) & \cong \bigoplus_{\alpha \in \ap}  \cSz(\Wone; K, L) \otimes \KhRN(J)\\
 & \cong \bigoplus_{\alpha \in \ap}   \cSz(\Wone; K, L) \otimes \Z[X]/(X^N).
\end{align*}
 We used here the fact that $J$ is split disjoint from all
the attaching links for the \twoh-handles, and therefore each summand that appears
in the definition of $\cSz(\Wone; K, L\cup J)$ splits off a $\KhRN(J)$ factor;
moreover, the equivalence relation is compatible with this splitting.

The map $\Psi_{I \times Y ; \Delta_+, \alpha'}$ is now easy to describe. It is
induced by capping $J$ with a disk, so it only affects the factor $\KhRN(J)$, in
a standard way. Precisely, we have
\begin{equation}
\label{eq:psiplus}
 \Psi_{I \times Y ; \Delta_+, \alpha'}(v \otimes X^n) = \begin{cases}
v & \text{if } n=N-1,\\
0 & \text{if } n =0, 1, \dots, N-2,
\end{cases}
\end{equation}
for all $v \in \cSz(\Wone; K, L)$.

To describe the second map $\Psi_{I \times Y ; \Delta_-, \alpha'}$, consider the
diagram
\begin{equation}
\label{eq:bigdiagram}
\begin{tikzcd}[column sep=huge] 
\underset{\alpha \in \ap}{\bigoplus}  \Sz(\Wone; K(k^-, k^+) \cup L \cup J, \alpha) \arrow[r, "\Psi_{I \times \del \Wone; \Sigma_-, \alpha'}"]\arrow[d, start anchor={[yshift=10pt]south}] &\underset{\alpha \in \ap}{\bigoplus}  \Sz(\Wone; K(k^-+s^-, k^++s^+) \cup L, \alpha) \arrow[d, start anchor={[yshift=10pt]south}]\\
\underset{\alpha \in \ap}{\bigoplus}  \cSz(\Wone; K, L \cup J) \arrow[r, dashed, "\bPsi_{I \times \del \Wone; \Sigma_-, \alpha'}"] \arrow[d, start anchor={[yshift=10pt]south}, "\Phi"', "\cong"] & \underset{\alpha \in \ap}{\bigoplus} \cSz(\Wone; K, L)  \arrow[d,start anchor={[yshift=10pt]south}, "\Phi"', "\cong"] \\
\underset{\alpha \in \ap}{\bigoplus} \Sz(\Wtwo; L \cup J, \alpha) \arrow[r, "\Psi_{I \times Y ; \Delta_-, \alpha'}"] & \underset{\alpha \in \ap}{\bigoplus}  \Sz(\Wtwo; L, \alpha).
\end{tikzcd}
\end{equation}
Here, in the top row we wrote $(k^-, k^+)$ for a pair $(r-\alpha^-, r-\alpha^+)$
as in Definition~\ref{def:cabled}. The vertical maps from the first to the
second row are induced by the inclusion of the summands into the cabled skein
lasagna module; cf. Definition~\ref{def:cabled}. The vertical maps from the
second to the third row are the isomorphisms $\Phi$ from Theorem~\ref{thm:12h}. 

Ignoring the middle dashed arrow for the moment, note that the above diagram
commutes. Indeed, by the definition of $\Phi$ in the proof of
Theorem~\ref{thm:12h}, the vertical compositions (from the first to the third
row) are given by attaching cores of the \twoh-handles to lasagna fillings in $\Wone$.
Note that we are attaching more cores on the right; namely, those in the
boundary of $\del \Sigma$, counted by the vectors $s^-$ and $s^+$. The
horizontal cobordism maps (as defined in Section~\ref{sec:gluing}) are given by
attaching the surface $\Sigma_-$ (in the top row) and $\Delta_-$ (in the bottom
row). Because $\Delta_-$ is the union of $\Sigma_-$ and the extra cores of
\twoh-handles counted by $s^-$ and $s^+$,  the diagram \eqref{eq:bigdiagram}
commutes.

Since the bottom vertical arrows in the diagram are isomorphisms, let us now add
the middle dashed arrow, given by the map
$$ \bPsi_{I \times \del \Wone; \Sigma_-, \alpha'}:= \Phi^{-1}\circ \Psi_{I \times Y ;
\Delta_-, \alpha'} \circ \Phi.$$ Because \eqref{eq:bigdiagram}  commutes, we
deduce that this map is induced on the skein lasagna modules by applying the
cobordism maps $\Psi_{I \times \del \Wone; \Sigma_-, \alpha'}$ on each summand; this
justifies the notation.

Recall that $\Sigma_-$ is the complement of the disk $\Delta_+$ inside $\Sigma$.
Thus, we can write the cobordism maps $\Psi_{I \times \del \Wone; \Sigma_-, \alpha'}$
in terms of the maps $\Psi_{I \times \del \Wone; \Sigma(n\bullet), \alpha'}$
associated to the surface $\Sigma$ with $n$ dots, as in \eqref{eq:dots}:
$$\Psi_{I \times \del \Wone; \Sigma_-, \alpha'}(v \otimes X^n) = \Psi_{I \times \del
\Wone; \Sigma(n\bullet), \alpha'}(v).$$

Fixing $n$, the maps $\Psi_{I \times \del \Wone; \Sigma(n\bullet), \alpha'}$ on
various summands in the construction of the skein lasagna module induce a map:
$$ \bPsi_{I \times \del \Wone; \Sigma(n\bullet), \alpha'}\colon \bigoplus_{\alpha \in \ap}  \cSz(\Wone; K, L) \to \bigoplus_{\alpha \in \ap}   \cSz(\Wone;K, L)$$ such that
\begin{equation}
\label{eq:psiminus}
 \bPsi_{I \times \del \Wone; \Sigma_-, \alpha'}(v \otimes X^n) = \bPsi_{I \times \del \Wone; \Sigma(n\bullet), \alpha'}(v).
 \end{equation}

We are now ready to give a general formula for the skein lasagna module of a
\fourm-manifold decomposed into handles in terms of skein lasagna modules of
\oneh-handlebodies. We will phrase it for an arbitrary number of handles.

\begin{theorem}
\label{thm:main}
Consider \fourm-manifolds
$ \Wone \subseteq \Wtwo \subseteq \Wthree \subseteq \Wfour$ where \begin{itemize}
\item $\Wone=\natural^m(S^1 \times B^3)$ is the union of $m$ \oneh-handles;
\item  $\Wtwo$ is obtained from $\Wone$ by attaching $n$ two-handles along a framed link
$K$;
\item $\Wthree$ is obtained from $\Wtwo$ by attaching $p$ three-handles along spheres
$S_1, \dots S_p$;
\item $\Wfour$ is obtained from $\Wthree$ by attaching some four-handles.
\end{itemize}

Consider also a framed link $L \subset \del \Wfour$. We represent $\Wfour$ by a Kirby
diagram, viewing $K \cup L$ as a link in $\del \Wone$, and the spheres $S_i$ in
terms of surfaces $\Sigma_j$ on $\del \Wone$ with $\del \Sigma_j$ consisting of some
copies of various components of $K$ (so that $S_j$ is the union of $\Sigma_j$
and the corresponding cores of the \twoh-handles). 

Given 
$$\alpha ' \in H_2^L(\Wfour; \Z) \cong H_2^L(\Wthree; \Z) \cong H_2^L(\Wtwo; \Z)/([S_1],
\dots, [S_p]),$$ let $\ap$ be the set of all $\alpha \in H_2^L(\Wtwo; \Z) \subseteq
\Z^n$ whose equivalence class modulo $([S_1], \dots, [S_p])$ is $\alpha'$.  

Then, the skein lasagna module $ \Sz(\Wfour; L, \alpha')$ is isomorphic to the
quotient of the direct sum of cabled skein lasagna modules $\bigoplus_{\alpha \in \ap}  \cSz(\Wone; K, L)$ by the relations
\begin{equation}
\label{eq:P1}
\bPsi_{I \times \del \Wone; \Sigma_j(n\bullet), \alpha'}(v)=0, \ \  n =0, 1, \dots, N-2,
\end{equation}
and
\begin{equation}
\label{eq:P2}
 \bPsi_{I \times \del \Wone; \Sigma_j((N-1)\bullet), \alpha'}(v)= v
\end{equation}
for all $v \in \bigoplus_{\alpha \in \ap}   \cSz(\Wone; K, L)$ and $j=1, \dots, p$.
\end{theorem}

\begin{proof}
First, note that the addition of \fourh-handles does not affect the skein lasagna
module, in view of Proposition~\ref{prop:threefour}. Thus, we can consider $\Wthree$
instead of $\Wfour$.

The skein lasagna module of $L$ viewed in the boundary of $\del W$ is given by
$\cSz(\Wone; K, L)$ according to Theorem~\ref{thm:12h}. When we add a \threeh-handle, we
divide by the relations 
\begin{equation}
\label{eq:psipsi}
\Psi_{I \times Y ; \Delta_+, \alpha'}(v \otimes X^n)= \Psi_{I \times Y; \Delta_-, \alpha'}(v \otimes X^n),
\end{equation} 
as proved in Theorem~\ref{thm:3h}. In terms of the identifications $\Phi$ from
Theorem~\ref{thm:12h}, the left hand side of \eqref{eq:psipsi} is given by
Equation~\eqref{eq:psiplus}, and the right hand side by
Equation~\eqref{eq:psiminus}. We thus get relations of the form \eqref{eq:P1}
and \eqref{eq:P2}. The generalization to multiple \threeh-handles is
straightforward.
\end{proof}

Theorem~\ref{thm:main} gives a description of an arbitrary skein lasagna module
in terms of skein lasagna modules for links in the boundary of $\Wone=\natural^m(S^1
\times B^3)$, and cobordism maps for surfaces in $I \times \del \Wone$. In the next
section we will obtain a further reduction to links in $S^3$ and cobordism maps
between them, under the additional constraint of working with field
coefficients; see Theorem~\ref{thm:mainonehandle}.

 %This file explains how to get the skein module from a Kirby diagram
%!TEX root=skein.tex

\section{One-handles}
\label{sec:1h}

Consider \fourm-manifolds $\Wgeneric$ and $\Whandle$, where $\Whandle$ is the result of attaching a
finite number of \oneh-handles to $\Wgeneric$. The boundary of the cocore of each
\oneh-handle is a $2$-dimensional sphere $S^2\subset \del \Whandle$ that generically
intersects links $L\subset \del \Whandle$ in a finite set of points. In this section
we aim to compute $\Sz(\Whandle;L)$ in terms of the invariants $\Sz(\Wgeneric;R\cup
\bigsqcup_i (T_i \sqcup \overline{T_i}))$ of
the \fourm-manifold $\Wgeneric$ and some links $R\cup
\bigsqcup_i (T_i \sqcup \overline{T_i})\subset \del \Wgeneric$ related to $L$.

Throughout this section we will work with coefficients in a field $\k$. Under
this assumption $\KhR_N$ is strictly monoidal under disjoint union (without
$\mathrm{Tor}$ terms) and sends mirror links to dual link homologies (without
$\mathrm{Ext}$ terms). As a consequence, $\Sz$ is monoidal under (boundary)
connect sum; see \cite[Theorem 1.4 and Corollary 7.3]{MN}. We leave the
investigation of the behavior under more general coefficient rings to future
work.

\subsection{One-handles away from links}
We first consider the case when $L$ is disjoint from the cocores of the
\oneh-handles. Up to a small isotopy, we may even assume that $L$ is disjoint from the
entire boundary of the added \oneh-handles, i.e. that $L\subset \del \Wgeneric$. As in
Proposition~\ref{prop:threefour}, the corresponding invariants are related by a
canonical map and we have:

\begin{lemma}\label{lem:onehaway} The inclusion $i\colon (\Wgeneric,L) \to (\Whandle,L)$ induces an isomorphism
    \[i_*\colon \Sz(\Wgeneric;L,\k) \xrightarrow{\cong} \Sz(\Whandle;L,\k)\]
\end{lemma}
\begin{proof}
    The proof is a straightforward generalization of the proof of \cite[Theorem
    1.4]{MN}, which deals with boundary connected sums. The map $i_*$ is induced
    by the map sending lasagna fillings of $(\Wgeneric,L)$ to lasagna fillings of
    $(\Whandle,L)$ along the embedding $i$. The inverse is given on lasagna fillings
    $F$ in $(\Whandle,L)$ by looking at their intersection with a neighborhood of the
    cocores of all \oneh-handles. Up to a small isotopy, each such intersection is
    an identity cobordism on a link $K\subset B^3$. The inverse map is given by
    replacing it by a sum of pairs of input balls, labelled by basis and dual
    basis elements of $\KhR_N(K)$ respectively. The resulting linear
    combination of fillings can be isotoped into $\Wgeneric$, and is equivalent to the
    original filling according to the neck-cutting lemma (Lemma 7.2 in
    \cite{MN}).
\end{proof}

\begin{corollary}
    \label{cor:SoneBthree}
    There are canonical isomorphisms 
    \[\k \xrightarrow{\cong} \Sz(S^1 \times B^3;\emptyset,\k), \qquad \k
    \xrightarrow{\cong} \Sz(S^1 \times S^3,\k)\] each sending $1\in \k$ to the
    respective empty lasagna filling.  
\end{corollary}
\begin{proof} The first isomorphism is given by a \oneh-handle attachment to
$(B^4,\emptyset)$ as in Lemma~\ref{lem:onehaway}. The second isomorphism can be
proved similarly: Let $F$ be a lasagna filling of $S^1 \times S^3$ and consider
its intersection with a fiber $\{x\}\times S^3$. Up to a small isotopy, we may
assume that the filling $F$ intersects $\{x\}\times S^3$ transversely (in lasagna sheet,
not in input balls) and disjointly from $\{x\}\times \{\text{north pole}\}$.
Then for small $\epsilon>0$, the intersection $F\cap [x-\epsilon,
x+\epsilon]\times (S^3\setminus \text{north pole})$ is an identity cobordism on
a link $K$. We replace this by a sum over pairs of input balls labelled with
basis and dual basis elements of $\KhR_N(K)$ respectively. The resulting closed
lasagna filling is supported in a single $B^4$ and can, thus, be identified with
a scalar multiple of the empty filling. 
\end{proof}

\begin{remark}
    It is instructive to evaluate the inverse to the canonical isomorphisms from
    Corollary~\ref{cor:SoneBthree} on surfaces of revolution generated by links.
    Any framed, oriented link $K\subset B^3$ or $S^3$ defines a
    vegetarian\footnote{A lasagna filling consisting only of a surface, without
    input \emph{meat balls}.} lasagna filling $S^1 \times K$ of $S^1 \times
    B^3$, which evaluates to a scalar multiple of the empty lasagna filling. It
    follows from the proofs of Lemma~\ref{lem:onehaway} and
    Corollary~\ref{cor:SoneBthree} that this scalar is the trace of the identity
    map on $\KhR_N(K)$.  Here it is important to take the Koszul signs in the
    symmetric monoidal structure on (homologically and quantum) bigraded vector
    spaces into account. The trace is thus $\tr(\Id_{\KhR_N(K)}) =
    \chi_{q=1}(\KhR_N(K)) =  \pm N^{|\pi_0(K)|}$ , i.e. the $\gl_{N}$ quantum
    link polynomial of $K$, specialized at $q=1$. More generally, any
    endocobordism of $K$ defines a lasagna filling of $S^1\times B^3$ that is a
    multiple of the empty filling, with coefficient given by the graded trace of
    the induced endomorphism of $\KhR_N(K)$; see e.g.~\cite[Section 6]{Jac},
    \cite[Section 10.1]{MR2174270}, \cite[Theorem D]{BPW} for related discussions of
    Lefschetz traces in the case of Khovanov homology.
\end{remark}

\subsection{Cutting and gluing 1-handles}

Consider the process of \emph{cutting a lasagna filling} $F$ of $\Wone =
\natural^m(S^1 \times B^3)$ with boundary $L$ along the cocores $C_i \cong
\mathrm{pt} \times B^3$ of the \oneh-handles for $1\leq i\leq m$. Let us assume
that the lasagna sheet $\Sigma$ of $F$ intersects the cocores transversely in
tangles $T_i:=\Sigma\cap C_i$. In particular, the link $L$ intersects the belt
spheres $S_i:=\del C_i$ geometrically in $2p_i$ points, the boundary points of
the tangle $T_i$. The algebraic intersection numbers are all zero, since $L$ is
null-homologous, as witnessed by $F$. In this way, we obtain a lasagna filling
$\mathrm{cut}(F)$ of $\Wone\setminus \bigsqcup_i n(C_i)\cong B^4$ with boundary link
$$L_T:= (L\setminus \bigsqcup_i(L \cap S_i)) \cup (T_i \cup \overline{T_i}).$$
The latter is obtained by cutting $L$ open at the $2p_i$-tuples of boundary
points and inserting copies of the tangles $T_i$ and $\overline{T_i}$,
schematically: 
\[
    \begin{tikzpicture}[anchorbase,scale=.2]
        \bgsphere{0}{0}{3}
        \fgsphere{0}{0}{3}
        \bentline{1}{0}{thick}
        \bentline{1}{-6}{thick} 
        \opentrefoil{0}{-.5}{1}{1}{orange, thick, dotted}
        \bentline{1}{-4.8}{orange}
        \node at (0,2) {\tiny${\color{red}\bullet}$};
        \node at (-1,1.5) {\tiny${\color{red}\bullet}$};
        \bentline{1}{-1}{thick,red}
        \bentline{1}{-1.5}{thick,red}
        \end{tikzpicture}   
        \mapsto
        \begin{tikzpicture}[anchorbase,scale=.2]
            \sphere{0}{0}{3}
            \sphere{7}{0}{3}
            \bentlinegap{1}{0}{thick}{0}
            \bentlinegap{1}{-6}{thick}{0} 
            \opentrefoil{0}{-.5}{1}{1}{thick, red}
            \opentrefoil{7}{-.5}{1}{-1}{thick, red}
            \bentlinegap{1}{-4.8}{orange}{0} 
            \bentlinegap{1}{-1}{thick,red}{0}
            \bentlinegaps{1}{-1.5}{thick,red}{0}
            \node at (0,2) {\tiny${\color{red}\bullet}$};
            \node at (-1,1.5) {\tiny${\color{red}\bullet}$};
            \node at (7,2) {\tiny${\color{red}\bullet}$};
            \node at (8,1.5) {\tiny${\color{red}\bullet}$};  
            \end{tikzpicture}    
\]

Of course, the procedure of cutting lasagna fillings does not
describe a well-defined map on the level of $\Sz$ since it does not respect the
skein relations. Instead we consider the reverse operation. 

The process of \emph{gluing a lasagna filling} works as follows. Let $F'$ be a
lasagna filling of $B^4$ with boundary link $L_T$ as above; i.e., inside
$S^3=\del B^4$ we have $m$ pairs of embedded 3-balls $B_i\cup \overline{B_i}$,
such that $L_T \cap B_i = T_i$ and $L_T\cap \overline{B_i} = \overline{T}_i$ for
$1\leq i\leq m$. Denote the numbers of boundary points by $2p_i:=|\partial
T_i|$. Now we attach $m$ \oneh-handles with core-parallel lasagna sheets $I \times
T_i \subset I \times B^3$ along the $B_i\cup \overline{B_i}\cong S^0 \times B^3$
to obtain a lasagna filling of $\Wone$ with boundary $L$. Since the relations in
$\Sz$ are local, this induces a map:
\begin{equation}
    \label{eq:1hgluing}
\mathrm{glue}_{L_T}\colon \Sz(B^4; L_T,\k)
\{(\textstyle\sum_i p_i)(N-1)\}  
\to \Sz(\Wone; L,\k)
\end{equation}
The grading shift is there to compensate the change in Euler characteristic of
the surfaces in lasagna fillings upon gluing.

\begin{lemma}\label{lem:surj} For every lasagna filling $F$ of $\Sz(\Wone; L,\k)$, there exists a
framed $L_T\subset \del B^4$, such that $F$ is contained in the image of $\mathrm{glue}_{L_T}$.
\end{lemma}
\begin{proof}
    By a small isotopy, we may assume that $F$ satisfies the assumption of the
    cutting procedure described above. The statement now follows since cutting,
    albeit ill-defined, is manifestly a right-inverse to gluing.
\end{proof}

It follows that the gluing maps from \eqref{eq:1hgluing} assemble to a
surjective map from a direct sum of shifts of $\Sz(B^4; L_T,\k)$ to $\Sz(\Wone; L,\k)$. Here, the
sum is indexed by all ways of writing $L$ as a contraction of links $L_T$
obtained by drilling out pairs of tangles $T_i\cup \overline{T_i}$ and resealing
the boundary points across the \oneh-handles. It remains to describe the kernel. 

\begin{definition} 
    \label{def:threeball}
    For $p\in \N$ fix a configuration $P_p$ of $2p$ framed points in $S^2=\del
    B^3$, partitioned into two halves with opposite co-orientations. We define a
    category $\Sz(B^3;P_p)$ enriched in bigraded $\k$-vector spaces with:
    \begin{itemize}
        \item objects: framed, oriented tangles $T$ in $(B^3;P_p)$ inducing the
        given orientation on $P_p$
        \item morphisms given by
        \begin{align}\Hom_{\Sz(B^3;P_p,\k)}(T_1,T_2)&:= 
        \KhRN(T_2\cup_{P_p} \overline{T_1},\k)\{p(N-1)\} \label{eq:KhRmor} \\ &= \Sz(B^4; T_2\cup_{P_p}
        \overline{T_1},\k)\{p(N-1)\} \label{eq:skeinmor}
        \end{align}
    \end{itemize}
    with (grading-preserving) composition maps induced in the case of the
    right-hand side of \eqref{eq:KhRmor} by the action of merging cobordisms, as
    described in \cite[Section 6.1 (vertical composition of 2-morphisms)]{MWW},
    and in the case of \eqref{eq:skeinmor} induced by the gluing of lasagna
    fillings of balls.
\end{definition}

\begin{lemma} \label{lemma:module} Let $\Wgeneric$ be a smooth, oriented, connected, compact \fourm-manifold. Fix
$B^3 \subset \del \Wgeneric$ and consider a link $L_1$ that intersects $B^3$ in a tangle
$T_1$ with boundary $\del T_1=P_p$, i.e. $L_1=R\cup_{P_p} T_1$. Now let $T_2$ be
another such tangle and $L_2 = R \cup_{P_p} T_2$, then we have a grading-preserving \emph{gluing
map}
$$
\Sz(\Wgeneric; L_1,\k) \otimes 
\Sz(B^4; T_2\cup_{P_p}\overline{T_1},\k)\{p(N-1)\} 
\to \Sz(\Wgeneric; L_2,\k).
$$
Moreover, these gluing maps are compatible with composition in $\Sz(B^3;P_p,\k)$
in the sense that all diagrams of the following type commute:
\[
    \begin{tikzcd}[column sep=-2.45in,row sep=.2in] 
        &\Sz(\Wgeneric; L_1,\k) 
        \otimes 
        \Sz(B^4; T_2\cup_{P_p}\overline{T_1},\k)
        \otimes 
        \Sz(B^4; T_3\cup_{P_p}\overline{T_2},\k)\{2p(N-1)\} 
            \arrow[dr]\arrow[dl]
    & 
    \\
        \Sz(\Wgeneric; L_2,\k) 
        \otimes 
        \Sz(B^4; T_3\cup_{P_p}\overline{T_2},\k)\{p(N-1)\} 
            \arrow[dr]
    &  &
        \Sz(\Wgeneric; L_1,\k) 
        \otimes 
        \Sz(B^4; T_3\cup_{P_p}\overline{T_1},\k)\{p(N-1)\}   
            \arrow[dl]
    \\        
        & \Sz(\Wgeneric; L_3,\k) &
        \end{tikzcd}    
\]
\end{lemma}
\begin{proof}
    Straightforward on the level of lasagna fillings. The map descends to the
    quotient since skein relations are local.
\end{proof}

The statement of Lemma~\ref{lemma:module} can be paraphrased as: the choice of a
$3$-ball with point configuration $P_p$ in $\del \Wgeneric$ equips $\Sz(\Wgeneric; -,
\k):=\bigoplus_{L}\Sz(\Wgeneric, L, \k)$ with the structure of a bigraded module for the
category $\Sz(B^3; P_p)$. (Here the direct sum is taken over all links $L$ that
intersect the boundary of the chosen $3$-ball in the fixed configuration $P_p$.)

\begin{theorem}
    \label{thm:mainonehandle}
Let $\Wone = \natural^m(S^1 \times B^3)$ with a nullhomologous link $L\subset \del
\Wone$ in the boundary that intersects the belt spheres of the \oneh-handles
transversely in $2p_i$ points for $1\leq i\leq m$. Let $R\subset
S^3 \setminus \bigsqcup_i (B_i\cup \overline{B_i})$ denote the tangle obtained
from $L$ by cutting open along the belt spheres. Then we have an isomorphism:
$$
\bigoplus_{\substack{\mathrm{tangles } ~ T_i \\|\del T_i|=2p_i}} 
\KhRN(R\cup \bigsqcup_i (T_i \sqcup \overline{T_i}), \k) 
\{(\textstyle\sum_i p_i)(N-1)\}  
\big/ \sim \; 
\xrightarrow{\cong} \; \Sz(\Wone; L,\k)
$$
where the relation $\sim$ is given by taking coinvariants for the actions of
$\Sz(B^3;P_{p_i},\k)$, i.e. by identifying the images of the actions
\[
\xymatrix@C=-5em{
    &  
    \KhRN(R\cup \bigsqcup_i (T_i \sqcup \overline{T'_i}),\k) \otimes \bigotimes_i \KhRN(T'_i \cup T_i,\k) \{p_i(N-1)\}\ar[dl]\ar[dr] 
    &   
    \\
    \KhRN(R\cup \bigsqcup_i (T_i \sqcup \overline{T_i}),\k) 
    &  
    & 
    \KhRN(R\cup \bigsqcup_i (T'_i \sqcup \overline{T'_i}),\k)  }
\]
for all pairs of tangles $T_i$, $T'_i$ with boundary $P_{p_i}$.
(Here we have omitted a global grading shift.)
\end{theorem}
\begin{proof} The map is defined by first considering the direct sum of the gluing morphisms 
 \[\KhRN(R\cup \bigsqcup_i (T_i \sqcup \overline{T_i}),\k)\{(\textstyle\sum_i p_i)(N-1)\} \to \Sz(\Wone; L, \k) \] from
    \eqref{eq:1hgluing}. The coinvariants for the actions of $\Sz(B^3;P_{p_i},\k)$
    clearly lie in the kernel, so we get an induced map from the indicated
    quotient to $\Sz(\Wone; L,\k)$, which we again call the gluing map. It is
    surjective by Lemma~\ref{lem:surj}, so it remains to prove injectivity.
    
     Let
    $F_1, F_2$ be two equivalent linear combinations of lasagna fillings in
    $\Sz(\Wone; L,\k)$, and let $G_1,G_2$ be respective preimages under the gluing map. We want 
    to show that $G_1$ and $G_2$ are equivalent. Without loss of generality, we
    may assume that $F_1$ and $F_2$ are individual lasagna fillings (rather than
    linear combinations) and that they differ by a single move as in
    Lemma~\ref{lem:fixballs} with the relevant input ball fixed and disjoint
    from the cocores of the \oneh-handles in $\Wone$. If $F_1$ and $F_2$ differ by a
    replacement inside the fixed input ball or an isotopy supported away from
    the cocores, then $G_1$ and $G_2$ are equal in $\KhRN(R\cup \bigsqcup_i (T_i
    \sqcup \overline{T_i}),\k)$. If $F_1$ and $F_2$ differ by an isotopy supported
    in a neighborhood of the cocores, then $G_1$ or $G_2$ differ by an element
    of the subspace factored out. Since every isotopy of lasagna fillings can be
    factored in this way, we get that $G_1$ and $G_2$ are equivalent.
\end{proof}

Theorem~\ref{thm:mainonehandle} can also be summarized by saying that $\Sz(\Wone;
L,\k)$ is computed by the zeroth Hochschild homology of a tensor product of
$3$-ball categories, namely one for each handle, with coefficients in a bimodule
associated to the tangle $R$ that results from $L$ by cutting open along the
belt spheres. We will discuss the details of this perspective in a special case
in Section~\ref{subsec:threeballalg}. 

\begin{remark}
    Similarly to the \twoh-handle formula from Theorem~\ref{thm:12h}, the \oneh-handle
    formula from Theorem~\ref{thm:mainonehandle} expresses the skein module of
    the more complicated manifold as a quotient of a (countable) direct sum of
    invariants of simpler manifolds. A possibly relevant difference, however, is
    that the \twoh-handle formula features only finitely many summands with a
    given shift in quantum grading, whereas this number is infinite for the
    \oneh-handle formula. 
    
    The skein modules that have been computed using only the \twoh-handle formula,
    first and foremost in \cite{MN}, are locally finite-dimensional, i.e.
    finite-dimensional in each bidegree. It is an open question whether this is
    true for all \fourm-manifolds admitting handle decompositions without
    \oneh-handles. In the rest of this paper we will see that local
    finite-dimensionality may fail when \oneh-handles are present.
\end{remark}

Finally we comment on the functoriality of the \oneh-handle formula from
Theorem~\ref{thm:mainonehandle}. We have seen
that $\Sz(\Wone; L,\k)$ for $\Wone=\natural^m(S^1 \times B^3)$ is a colimit of link
homologies for links in $S^3$, which result from cutting $L$ along belt spheres
and inserting pairs of tangles. Now consider a link cobordism $S\subset \del \Wone
\times I=:Z$ from $L\subset \Wone$ to $L' \subset \Wone'$ where $\Wone' = \Wone \cup Z$. We
claim that the induced map $$\Psi_{Z;S}\colon  \Sz(\Wone; L,\k) \to \Sz(\Wone'; L',\k)$$
can also be expressed in terms of cobordism maps between links in $S^3$. Recall
that the cobordism map $\Psi_{Z;S}$ sends a lasagna filling $F$ of $\Wone$ to the
composite lasagna filling $F\cup S$ of $\Wone \cup Z$. In a generic situation,
cutting the cocores has the following local model. Here we display the filling
$F$ in the inner tube and $S$ in the outer, spherical shell.   
\[
    \begin{tikzpicture}[anchorbase,scale=.2]
        \bgsphere{0}{0}{3}
        \fgsphere{0}{0}{3}
        \bgsphere{0}{0}{5}
        \fgsphere{0}{0}{5}
        \bentline{1}{0}{thick,opacity=.3}
        \bentline{1}{-6}{thick,opacity=.3} 
        \bentline{1}{2}{thick}
        \bentline{1}{-8}{thick} 
        \opentrefoil{0}{-.5}{1}{1}{orange, thick, dotted}
        \bentline{1}{-4.8}{orange}
        \node[opacity=.5] at (0,2) {\tiny${\color{red}\bullet}$};
        \node[opacity=.5] at (-1,1.5) {\tiny${\color{red}\bullet}$};
        \bentline{1}{-1}{thick,red,opacity=.5}
        \bentline{1}{-1.5}{thick,red,opacity=.5}
        \bentline{1}{1}{thick,blue}
        \bentline{1}{0.5}{thick,blue}
        \node at (0,4) {\tiny${\color{blue}\bullet}$};
        \node at (-1,3.5) {\tiny${\color{blue}\bullet}$};
        \draw[blue, thick, dotted] (0,2) \pu (-1,3.5);  
        \draw[blue, thick, dotted] (-1,1.5) \pu (0,4);  
        \end{tikzpicture}   
        \mapsto
        \begin{tikzpicture}[anchorbase,scale=.2]
            \bgsphere{0}{0}{3}
            \fgsphere{0}{0}{3}
            \bgsphere{12}{0}{3}
            \fgsphere{12}{0}{3}
            \sphere{0}{0}{5}
            \sphere{12}{0}{5}
            \bentlinegap{1}{0}{thick,opacity=.3}{5}
            \bentlinegap{1}{-6}{thick,opacity=.3}{5} 
            \bentlinegap{1}{2}{thick}{5}
            \bentlinegap{1}{-8}{thick}{5} 
            \opentrefoil{0}{-.5}{1}{1}{thick, red}
            \opentrefoil{12}{-.5}{1}{-1}{thick, red}
            \bentlinegap{1}{-4.8}{orange}{5} 
            \bentlinegap{1}{-1}{thick,red,opacity=.5}{5}
            \bentlinegaps{1}{-1.5}{thick,red,opacity=.5}{5}
            \node[opacity=.5] at (0,2) {\tiny${\color{red}\bullet}$};
            \node[opacity=.5] at (-1,1.5) {\tiny${\color{red}\bullet}$};
            \node[opacity=.5] at (12,2) {\tiny${\color{red}\bullet}$};
            \node[opacity=.5] at (13,1.5) {\tiny${\color{red}\bullet}$};  
            \bentlinegap{1}{1}{thick,blue}{5}
            \bentlinegaps{1}{0.5}{thick,blue}{5}
            \node at (0,4) {\tiny${\color{blue}\bullet}$};
            \node at (-1,3.5) {\tiny${\color{blue}\bullet}$};
            \node at (12,4) {\tiny${\color{blue}\bullet}$};
            \node at (13,3.5) {\tiny${\color{blue}\bullet}$};  
            \draw[blue, thick] (0,2) \pu (-1,3.5);  
            \node at (-.5,2.65) {\tiny${\color{white}\bullet}$};  
            \draw[blue, thick] (-1,1.5) \pu (0,4);  
            \draw[blue, thick] (12,2) \pu (13,3.5);  
            \node at (12.5,2.65) {\tiny${\color{white}\bullet}$};  
            \draw[blue, thick] (13,1.5) \pu (12,4);   
            \end{tikzpicture}    
\]
Let $S_i$ denote the tangle in $S^2 \times I$ that occurs as the intersection of
$S$ with the $i$th cocore and $R'\subset S^3 \setminus \bigsqcup_i (B_i\cup
\overline{B_i})$ the tangle obtained from $L'$ by cutting open along the belt
spheres of $\Wone'$. Denote by $2p'_i = |\del S_i| - 2p_i$ the number of outer
boundary point of $S_i$. Then the cobordism $\Sigma$ obtained from $S$ by
cutting along the annuli, which are the intersection of $Z$ with the cocores of
\oneh-handles in $\Wone'$, induces a cobordism map:
\[
\KhRN(R\cup \bigsqcup_i (T_i \sqcup \overline{T_i}),\k) \{(\textstyle\sum_i p_i)(N-1)\} 
\to 
\KhRN(R'\cup \bigsqcup_i ((S_i\cup T_i) \sqcup \overline{S_i\cup T_i}),\k)\{(\textstyle\sum_i p_i')(N-1)\} 
\]  
We claim that these components describe $\Psi_{Z;S}$ in terms of the colimit
formulas (left-hand sides) from Theorem~\ref{thm:mainonehandle}. To see this we
first observe that the unequal grading shifts guarantee that the components have
the same degree as $\Psi_{Z;S}$ (we have $\chi(\Sigma) =
\chi(S)+\sum_i(p_i+p'_i)$ and $\Sigma$ is glued to $\mathrm{cut}(F)$ along $p_i$
interval segments). Next we observe that after composing with the
projection-inclusion into the colimit formula for $\Sz(\Wone'; L', \k)$, the
resulting map no longer depends on the chosen location of cocores to cut.
Moreover, the subspace factored out in the colimit formula for $\Sz(\Wone; L, \k)$
is annihilated by the map thus defined. Thus the components described above
define a map $\Sz(\Wone; L, \k)\to \Sz(\Wone'; L', \k)$, and by construction this agrees
with $\Psi_{Z;S}$.

\subsection{Algebraic description of the 3-ball categories and their Hochschild homologies}
\label{subsec:threeballalg}
Recall the following definition, from e.g. \cite{MR3606445}.

\begin{definition}\label{def:HH}
    Let $K$ be a commutative ring and $\CC$ be a (small) $K$-linear
    category. Then the \emph{zeroth
    Hochschild homology} of $\CC$, also called the \emph{trace} of $\CC$, is defined as the $K$-module
    \[
    \HHz(\CC) := \Tr(\CC):= \left(\bigoplus_{x\in \mathrm{Ob}(\CC)} \End_\CC(x) \right) \bigg/ \mathrm{Span}\{f\circ g - g \circ f \}
    \]
    where the spanning set for the subspace to be divided out is constructed
    from all pairs of cyclically composable morphisms, i.e.  $f\in
    \Hom_\CC(x,y)$ and $g\in \Hom_\CC(y,x)$ for some $x,y\in \mathrm{Ob}(\CC)$.
\end{definition}

If $\CC$ as in Definition~\ref{def:HH} is not just enriched in $K$-modules, but
$M$-graded $K$-modules for some monoid $M$, then $\HHz(\CC)$ inherits the
structure of an $M$-graded $K$-module. The following is now an immediate consequence
of Theorem~\ref{thm:mainonehandle} and the Definitions~\ref{def:threeball} and
\ref{def:HH}.
\begin{corollary}\label{cor:HH}
    Let $\Wone=S^1\times B^3$ and consider the link $S^1\times P_p$ consisting of
    $2p$ parallel circles with balanced orientations (that is, with $p$
    circles oriented one way and $p$ the other way). Then, we have an
    isomorphism of bigraded $\k$-vector spaces:
    \begin{equation}
        \Sz(S^1\times B^3;S^1\times P_p,\k) \cong \HHz(\Sz(B^3;P_p,\k))
    \end{equation}
\end{corollary}

We now recall some facts about the zeroth
Hochschild homology, which we will use to show that the \oneh-handle formula may
compute vector spaces which are not locally finite-dimensional.

\begin{fact}
\label{fact:trace1}
Any functor $F\colon \CC \to \DD$ of $K$-linear categories induces natural
$K$-module homomorphism $\HHz(F)\colon\HHz(\CC)\to \HHz(\DD)$ sending
$[f\colon x\to x] \mapsto [F(f)\colon F(x)\to F(x)]]$. This is
well-defined since $f\circ g - g\circ f \mapsto F(f)\circ F(g)-F(g)\circ F(f)$.
If $F$ is an equivalence, then $\HHz(F)$ is an isomorphism; see e.g.
\cite{MR3606445}.  
\end{fact}

\begin{fact}
\label{fact:trace2}
Let $F\colon \CC \to \CC^{\oplus}$ and $G\colon \CC \to \mathrm{Kar}(\CC)$
denote the canonical embeddings of $\CC$ into its additive and its idempotent
completion, respectively. Then $\HHz(F)$ and $\HHz(G)$ are isomorphisms; see
e.g. \cite[Sections 3.4 and 3.5]{MR3606445}.
\end{fact}

In a slight reformulation of the functoriality results from \cite{ETW}, the tangle
invariant underlying the $\gl_N$ link homology over $\k$ can be described as a
2-functor:
\[\KhRbracket{-}\colon \Tang\xrightarrow{} H^\bullet(\Foamdg)\]
We now briefly explain the relevant algebraic structures here. 
\begin{itemize}
    \item As in \cite[Definition 6.1]{MWW} one defines a category $\mathbf{TD}$
    of \emph{tangle diagrams}, whose objects are finite words in the alphabet
    $\{\uparrow, \downarrow\}$ (which encode possible sequences of oriented
    boundary points for tangles) and whose morphisms are finite words in
    generating morphisms $\{\mathrm{cup}_i, \mathrm{cap}_i, \mathrm{crossing}_i,
    \mathrm{crossing}_i^{-1}\}$ (where the index $i$ specifies the strands
    participating in the generator), that are admissible in the sense that the
    composite describes a tangle diagram. The composition is concatenation of
    words. For details see \cite[Definition 6.1]{MWW}.
    \item  $\Tang$ is a $2$-category whose objects and $1$-morphisms are as in
    $\mathbf{TD}$. The $2$-morphisms are the framed, oriented tangle cobordisms
    in $[0,1]^4$ between standard lifts of tangle diagrams to actual tangles in
    $[0,1]^3$, considered up to isotopy rel boundary. 
    \item $\Foam$ is a (monoidal) $2$-category, enriched at the level of
    $2$-morphism spaces in $\k$-vector spaces and equipped with grading shift
    functors on $1$-morphisms. It has the same objects\footnote{More generally,
    one can consider labelled oriented points as objects in $\Foam$, but we will
    not need labels other than $1$.} as $\Tang$. The $1$-morphisms are (formal
    direct sums of grading shifts of) $\glN$ webs embedded in $[0,1]^2$ and the
    $2$-morphisms are (matrices with entries given by) $\k$-linear combinations
    of $\glN$ foams embedded in $[0,1]^3$, modulo certain local relations. For
    details see \cite{ETW}. 
    \item $\Foamdg$ is the (monoidal) $2$-category that is obtained from $\Foam$
    by replacing its $\k$-linear $\Hom$-categories by the corresponding dg
    categories. This means it has the same objects, but the $1$-morphisms are
    now chain complexes formed from $1$-morphisms in $\Foam$, where the
    differentials are given by $2$-morphisms in $\Foam$. The $2$-morphisms
    spaces are chain complexes of homologically homogeneous and quantum
    grading-preserving maps, spanned by $2$-morphisms from $\Foam$ (not
    necessarily chain maps). The differential on $2$-morphisms is the usual
    supercommutator with respect to the differential on the source- and target
    complexes. With respect to this differential the zero cycles are exactly the
    classical chain maps. There is also an enriched $2$-hom in $\Foamdg$, which
    is assembled from $2$-homs between objects shifted in quantum grading.
    \item $H^\bullet(\Foamdg)$ is the cohomology category of $\Foamdg$. It has
    the same objects and $1$-morphisms, but the $2$-morphism spaces are now
    graded $\k$-modules obtained by taking cohomology. The zeroth cohomology
    $H^0(\Foamdg)$ is also called the homotopy category; its $2$-morphisms are
    chain maps up to homotopy. 
    
    In the following we will also consider enriched $2$-homs.
    For objects $s,t$ and $1$-morphisms $A,B\colon s \to t$ we define the bigraded $\k$-modules:
    \begin{equation}
        \label{eq:enrichedhom}
        H^\bullet(\Foamdg)^*(A,B):=\bigoplus_{k\in \Z} \Hom_{H^\bullet(\Foamdg)}(A\{k\},B)
    \end{equation}
  Here one grading, the quantum grading, is
    given by the displayed direct sum, while the other grading, the homological
    grading, is already internal to $H^\bullet(\Foamdg)$. Using the grading
    shift automorphisms, these enriched $2$-homs admit composition maps and thus
    assemble into a bigraded $\k$-linear \emph{enriched morphism category}
    $H^\bullet(\Foamdg)^*(s,t)$ whose objects are the $1$-morphisms from $s$ to
    $t$.

    \item The functor $\KhRbracket{-}$ is the identity on objects. On
    $1$-morphisms it sends a tangle diagram to a chain complex of webs and foams
    in the way that is usual for $\glN$ link homology, and $2$-morphisms, i.e.\
    isotopy classes of tangle cobordisms are sent to the corresponding homotopy
    classes of chain maps as specified in the functoriality proof in \cite{ETW}. 
\end{itemize}

We recall from \cite[Section 6]{MWW} that the tangle invariant corresponding to
the $\glN$ link homology can be organized into a braided monoidal $2$-category.
Here we give a similar construction of this category $\Ta_N$ (which was denoted
$\mathbf{KhR}_N$ in\cite{MWW}) by replacing the top morphism layer of $\Tang$: 
\begin{itemize}
    \item objects are sequences of tangle endpoints, as in $\mathbf{TD}$ and $\Tang$,
    \item 1-morphisms consist of Morse data for tangles, as in $\mathbf{TD}$ and $\Tang$,
    \item 2-morphisms between tangles $S$ and $T$ with equal source and target
    objects are the bigraded $\k$-modules computed as the enriched 2-hom
    $H^\bullet(\Foamdg)^*(\KhRbracket{S},\KhRbracket{T})$ from
    \eqref{eq:enrichedhom} between the $\glN$ chain complexes of the tangles.
    \end{itemize}

As an important special case, one gets for a framed, oriented link $L$:
\[ \Hom_{\Ta_N}(\emptyset,L) \cong \KhRN(L).\] Moreover, if $T$ and $S$ are
framed, oriented tangles with endpoints identified, so that we can form the link
$T\cup \overline{S}$, then we set $2p=|\del S|=|\del T|$ and have:
\begin{align*} 
    \Hom_{\Ta_N}(S,T) 
    &\cong
 \Hom_{\Ta_N}(\emptyset,T\cup \overline{S}) \{p(N-1)\}
\cong \KhRN(T\cup \overline{S})\{p(N-1)\}
\end{align*}

Given a $3$-ball $B^3$ with a set $P_p$ of $2p$ framed, co-oriented points in the
boundary, together with a suitable identification of $(B^3,P_p)$ with $([0,1]^3,
s\cup t)$, we associate to it the morphism category $\Ta_N(s,t)$, whose objects
are tangles from $s$ to $t$. By construction, $\Ta_N(s,t)$ is equivalent to
$\Sz(B^3;P_p)$ from Definition~\ref{def:threeball}. Moreover, $\Ta_N(s,t)$ can be
considered as a full subcategory of the bigraded enriched morphism category
$H^\bullet(\Foamdg)^*(s,t)$.

\begin{remark}\label{rem:arcrings} For $N=2$ the foam $2$-category $\Foamtwo$ can be replaced by the
$2$-category (or canopolis) of Bar-Natan's dotted cobordisms \cite[Section
11.2]{MR2174270}; see \cite{2019arXiv190312194B}. The morphism categories of the
latter can also be described as categories of finitely-generated graded
projective modules for Khovanov's arc rings \cite{MR1928174}. 

\end{remark}

\subsection{The 3-ball category with two points}
Here we consider the categories from Section~\ref{subsec:threeballalg} in the
special case when the source and target objects consist of a single point
$s=t=\{*\}$. In this case, the corresponding morphism category in $\Foamdg$ is
known to be equivalent to the dg category of complexes of free graded
$R_N:=\k[X]/( X^N)$-modules; see e.g. \cite[Lemma 3.35]{MR3982970} for an
argument in an equivalent setting. We record this equivalence and its
consequence on the level of homology:
\[
    \Hom_{\Foamdg}(*,*) \simeq \Chdg(R_N\modgrfr), \quad
    \Hom_{H^\bullet(\Foamdg)}(*,*) \simeq H^\bullet(\Chdg(R_N\modgrfr))
\]
Here $R_N\modgrfr$ refers to the category of finitely-generated graded free
$R_N$-modules and $\Chdg(\CC)$ refers to the dg category of bounded chain
complexes over an additive category $\CC$. Again we will use a superscript $*$
to refer to the corresponding enriched morphism spaces, computed via the
ordinary morphism spaces between shifts of objects as in \eqref{eq:enrichedhom}.

Now we specialize to $N=2$ and classify the indecomposable objects. Setting
$R:=R_2=\k[X]/(X^2)$, the isomorphism classes of indecomposable objects (up to
shifts in quantum and homological degrees) in $H^\bullet(\Chdg(R\modgrfr))$ are
of the form: 
    \[C_k:= \underline{R} \xrightarrow{X}
    R\{-2\}\xrightarrow{X} \cdots  \xrightarrow{X} R\{-2k\}\]
     for $k\geq 0$; see \cite[Section 3]{KhovanovPatterns}. 

     Next we compute the zeroth Hochschild homology of
     $H^\bullet(\Chdg(R\modgrfr))$. In principle, there are two possible
     versions: using the ordinary or the
     enriched hom; see \cite[Section 2.4]{MR3653092}. In the case of the ordinary hom, we would obtain a
     $\Z$-graded (namely homologically graded) $\k[q^{\pm 1}]$-module, where $q$
     records the action of the auto-equivalence provided by the shift in quantum
     grading. We will, however, use the enriched hom (indicated by the
     superscript $*$) to consider the morphism spaces as bigraded. In doing so,
     one obtains \emph{translation} isomorphisms, which identify an object with
     all its gradings shifts. More specifically, between an object and its
     shift, the identity now represents an isomorphism of degree specified by
     the shift. The zeroth Hochschild homology of the resulting category carries
     the structure of a bigraded $\k$-vector space, since the endomorphism $q$
     now acts as the identity. 

\begin{proposition}\label{prop:HH-cxpolyring} The bigraded zeroth Hochschild
    homology of $H^\bullet(\Chdg(R\modgrfr))^*$ has a basis given by the trace
    classes $[\Id_{C_l}]$ and $[RX_{C_l}]$ for all $l\geq 0$. The identity
    morphisms on the complexes $C_l$ for $l\geq 0$ are self-explanatory and
    their trace classes have bidegree $(0,0)$. The endomorphism $RX_{C_l}$ is a
    special case $RX_{C_l} = RX^{(l)}_{C_l}$ of a larger family of endomorphisms
    $RX^{(l)}_{C_k}$ for $0\leq l\leq k$ of the following form:
    \begin{equation*}
        \begin{tikzcd}
            \underline{R} \arrow[r, "X"]&    
            \cdots  \arrow[r, "X"] & 
            R\{-2l\}\arrow[r, "X"] &
            \cdots  \arrow[r, "X"] & 
            R\{-2k\}\arrow[r, "0"]\arrow[d, "X"]& 
            \cdots  \arrow[r, "0"] & 
            0\\
            0 \arrow[r, "0"]&    
            \cdots  \arrow[r, "0"]& 
            R\{-2l-2\}\arrow[r, "X"]&
            \cdots  \arrow[r, "X"]& 
            R\{-2k-2\}\arrow[r, "X"]& 
            \cdots  \arrow[r, "X"]& 
            R\{-2k-2l-2\}
        \end{tikzcd}
        \end{equation*}
        where the only non-zero component is at $R\{-2k\}$ (which may coincide
        with $R\{-2l\}$ if $k=l$). The trace class of the morphism
        $RX^{(l)}_{C_k}$ has bidegree $(l,2l+2)$. ($RX$ stands for shift
        \emph{right} and apply $X$.) 
\end{proposition}
\begin{proof}
We abbreviate $\CC':=H^\bullet(\Chdg(R\modgrfr))^*$. Let $\CC$ denote the full
subcategory generated by the indecomposable objects $C_k$. By
Fact~\ref{fact:trace1} and the discussion of the beginning of the section, it
suffices to compute the bigraded zeroth Hochschild homology of $\CC$. To this
end, we study closed homogeneous endomorphisms of the objects $C_k$ and trace
relations between them. 

We note that the components of a chain map between shifts of such objects can
have quantum degree zero or two (a scalar multiple of $\Id_R$ or $X_R$). Since
the differential in every complex is of quantum degree two, this means that
closed morphisms with components of quantum degree zero are homotopic if and
only if they are equal. 

First we investigate the chain maps between shifts of objects $C_l$ with
components of quantum degree zero. For positive homological shifts (right shift)
there are simply no closed morphisms, i.e.\ no chain maps. In shift zero we have
the identity on every $C_l$ (which does not factor through any $C_m$ with $m\neq
l$) and for negative homological shifts we have closed maps that factor into a
composite of closed maps through a shift of a $C_m$ with $m<l$ (by induction,
one can show that their trace classes actually vanish). Thus in bidegree $(0,0)$
we have a basis of trace classes $[\Id_{C_l}]$ for $l\geq 0$. 

Second we are interested in chain maps between shifts of objects $C_l$ with
components of quantum degree two. In negative homological shifts (left shift)
all such maps are nullhomotopic. In non-negative homological shift, every such
map is homotopic to a scalar multiple of $RX^{(l)}_{C_k}$. However, one easily
checks that the trace class of $RX^{(l)}_{C_k}$ equals the trace class of $\pm
RX^{(l)}_{C_l}$. Since these have bidegree $(l,2l+2)$ in the enriched $\End$ of
$C_l$, we see that they are linearly independent. 
\end{proof}

Note that the bigraded zeroth Hochschild homology of
$H^\bullet(\Chdg(R\modgrfr))^*$ is \emph{not} locally finite-dimensional! It is of
countable dimension in bidegree $(0,0)$ with a basis given by $[\Id_{C_l}]$ for
$l\geq 0$. Nevertheless, we have:

\begin{proposition}\label{prop:twopoint} The bigraded vector spaces
   \[ \Sztwo(S^1\times B^3;S^1\times P_1,\k) \cong \HHz(\Sztwo(B^3;P_1,\k))
   \cong
   \HHz(\Ta_2(*,*)) \] 
   are four-dimensional, and in particular, locally finite-dimensional. 
\end{proposition}
\begin{proof}
    We have already explained the two isomorphisms. We now need to understand
    the essential image of $\Ta_2(*,*)$ under the full embedding into
    $H^\bullet(\Chdg(R\modgrfr))^*$. We claim that the invariant of any
    $(1,1)$-tangle decomposes into (shifts of) the indecomposable summands $C_0$
    and $C_1$, but never $C_l$ for $l\geq 2$. Provided this claim holds, we can compute
    $\HHz(\Ta_2(*,*))$ as the Hochschild homology of the full additive
    subcategory of $H^\bullet(\Chdg(R\modgrfr))^*$ generated by $C_0$ and $C_1$,
    and this again is isomorphic to the Hochschild homology of the full
    subcategory on the two objects $C_0$ and $C_1$. Here we use that the zeroth
    Hochschild homology is preserved under proceeding to the additive and
    idempotent completion; see Fact~\ref{fact:trace2}. Following the same
    arguments as in Proposition~\ref{prop:HH-cxpolyring}, we see that it is
    $4$-dimensional, spanned by $[\Id_{C_l}]$ and $[RX_{C_l}]$ for $l\in
    \{0,1\}$
    
    The key idea to prove the claim is that all complexes appearing in Khovanov
    homology come from complexes over $\k[X]$ by setting $X^2=0$ (though
    certainly not all complexes over $\k[X]/(X^2)$ have this property).  
    Indeed, one can use equivariant Khovanov homology, defined over the
    ring $\k[X,\alpha]/(X^2-\alpha) \cong \k[X]=:R'$ to simplify the complex of a
    $(1,1)$-tangle into a complex of graded free $\k[X]$-modules. These
    decompose, up to homotopy equivalence and shift, into chain complexes of the form 
    \[ C^0:=0 \xrightarrow{0} \underline{R'} \xrightarrow{0} 0, \quad \text{and}\quad 
    C^k:=\quad 0 \xrightarrow{0} \underline{R'} \xrightarrow{X^k} R'\{-2k\}
    \xrightarrow{0} 0 \quad \text{ for } k\geq 1\] Upon reducing to the ordinary
    Khovanov theory by tensoring with $\k[X]/(X^2)$ over $\k[X]$, these
    complexes decompose into (shifts of) copies of $C_0$ and $C_1$. 
\end{proof}

\begin{remark} A strong version of the so-called knight move conjecture posited
that the complex of any long knot decomposes (up to homotopy equivalence) into
one shifted copy of $C_0$ and some number of copies of $C_1$; see
\cite[Conjecture 1]{KhovanovPatterns}. The argument in the previous proof shows
that this can fail only due to the presence of \emph{more than one} shifted copy of $C_0$.
Three copies of $C_0$ can be detected in the counterexample to the knight move
conjecture found by Manolescu--Marengon \cite{MR4042864}.
\end{remark}

\begin{remark} One can also consider analogs of the skein modules $\Sz$ based on
equivariant or deformed versions of $\glN$ homology. For example, in one common
choice for $N=2$ one works over $R'=\k[X,\alpha]/(X^2=\alpha)$. We can also try
to compute the bigraded zeroth Hochschild homology of the 3-ball category with
two points and of its ambient category $H^\bullet(\Chdg(R'\modgrfr))^*$ in this setting. We have already
listed the indecomposable of the latter above: the chain complexes $C^k$. For $k\geq 1$ the
enriched isomorphism algebra of the complex $C^k$ is isomorphic to $R'[\eta]/(
X^k=0 )$ where $\eta$ is of bidegree $(1,2k)$. The trace classes of $\eta$ and
its multiples are zero. Moreover, the trace class of $X^x$ is zero for every
$x>0$. This leaves the trace classes of the identities of $C^k$ for $k\geq 0$
and the trace class of $X_{C^0}$ as linearly independent --- the zeroth
Hochschild homology is not locally finite-dimensional. However, it is currently
not known which $C^k$ appear in complexes of $(1,1)$-tangles. A copy of $C^3$
appears in \cite{MR4042864}.
\end{remark}

\subsection{The 3-ball category with four or more points}
We claim that the $3$-ball categories with $2p\geq 4$ points have zeroth
Hochschild homologies that are no longer locally finite-dimensional. Again we
restrict to the case of $N=2$ and work over a perfect field $\k$. Our strategy
is to give a lower bound for the dimension of the zeroth Hochschild homology in
terms of the split Grothendieck group. We briefly recall the relevant notions
and results.

\begin{definition}
\label{def:splitG}
 Let $\CC$ be an additive category. The \emph{split Grothendieck group} of $\CC$
 is defined as:
    \[
    K_0(\CC):=     \frac{\mathrm{Span}_\Z\{\text{isomorphism classes } [x] \text{ of objects in } \CC \}}{([x \oplus y] = [x] + [y] \mid x,y \in \mathrm{Ob}(\CC))}
    \]
   \end{definition}
  
  \begin{definition}
\label{def:KS}
    A $K$-linear additive category $\CC$ is called \emph{Krull--Schmidt}
    if every object decomposes uniquely into a finite direct sum of indecomposable
    objects with local endomorphism rings. 
    \end{definition}
   The following is clear from the definition:
   \begin{proposition}
   \label{prop:KS}
    For a Krull-Schmidt category, the split Grothendieck group is a free abelian
    group on the isomorphism classes of indecomposable objects in $\CC$.
    \end{proposition}
    
    \begin{definition} For a $K$-linear additive category $\CC$, the \emph{Chern character}
    is the $K$-linear map
    \[
        h_{\CC}\colon K_0(\CC) \otimes_\Z K \to  \HHz(\CC), \quad [x]\otimes 1 \mapsto [\Id_x\colon x \to x]
    \]
    \end{definition}
    
    \begin{proposition}[Proposition
    2.4 in \cite{MR3653092}]
    \label{prop:Chern}
    If $K=\k$ is a perfect field and $\CC$ is Krull-Schmidt with a
    finite-dimensional endomorphism algebra for each indecomposable object, then
    the Chern character $h_{\CC}$ is injective.
    \end{proposition}

Using these tools, we can now prove:

\begin{theorem} \label{thm:infdim} Let $p\geq 2$. Then $\Sztwo(S^1\times
B^3;S^1\times P_p,\k)$ is infinite-dimensional in bidegree $(0,0)$.
\end{theorem}
\begin{proof} We let $s=t= p\; \mathrm{points}$ and again have isomorphisms 
    \[ \Sztwo(S^1\times B^3;S^1\times P_p,\k) \cong \HHz(\Sztwo(B^3;P_p,\k))
    \cong \HHz(\Ta_2(s,t)) \] and we consider the category $\Ta_2(s,t)$ as a
    full subcategory of the enriched morphism category
    $H^\bullet(\Foamtwodg)^*(s,t)$.
    
    The $\k$-linear, additive category $H^\bullet(\Foamtwodg)^*(s,t)$ is
    Krull-Schmidt and hence idempotent complete; see e.g. the discussion in
    \cite[Sections 4.5, 4.8]{MR4380678} based on Bar-Natan's category, which is
    equivalent to $\Foamtwo$ by \cite{2019arXiv190312194B}.

    Now $\mathrm{Kar}(\Ta_2(s,t))^\oplus$ may be considered as an additive,
    idempotent complete full subcategory of $H^\bullet(\Foamtwodg)^*$; it is
    thus itself Krull--Schmidt.  We have $\HHz(\Ta_2(s,t))\cong
    \HHz(\mathrm{Kar}(\Ta_2(s,t))^\oplus)$ by Fact~\ref{fact:trace2}. Therefore, it suffices to compute its zeroth
    Hochschild homology of $\mathrm{Kar}(\Ta_2(s,t))^\oplus$.

    It is straightforward to check that the objects of
    $\mathrm{Kar}(\Ta_2(s,t))^\oplus$ have finite-dimensional endomorphism
    algebras, and since $\k$ is perfect, the Chern character \[h\colon
    K_0(\mathrm{Kar}(\Ta_2(s,t))^\oplus) \otimes_\Z \k \to
    \HHz(\mathrm{Kar}(\Ta_2(s,t))^\oplus)\] is injective; see
    Proposition~\ref{prop:Chern}. To prove that $\Sztwo(S^1\times B^3;S^1\times
    P_p,\k)$ is infinite-dimensional in bidegree $(0,0)$, it is thus sufficient
    to show that $K_0(\mathrm{Kar}(\Ta_2(s,t))^\oplus) \otimes_\Z \k$ is
    infinite-dimensional. 
    
    Moreover, $K_0(\mathrm{Kar}(\Ta_2(s,t))^\oplus)$ is free abelian on
    the isomorphism classes of its indecomposable objects; cf.
    Proposition~\ref{prop:KS}. Thus, we will be done once we can exhibit infinitely
    many indecomposable and pairwise non-isomorphic complexes appearing as
    (direct summands in) tangle complexes. 
    
    We will see that such complexes can be constructed as invariants of braids.
    Clearly, for $p\geq 2$ there are infinitely many braids on $p$ strands.
    Moreover, the braid complexes are invertible under tensoring with the
    complex for the respective inverse braid. Since the complex of the trivial
    braid is indecomposable (its endomorphism algebra
    $\left(\k[X]/(X^2)\right)^{\otimes p}$ is local), so are the complexes for
    all other braids. It is also known that all braid complexes are pairwise
    non-isomorphic. This can e.g. be deduced from the faithfulness of the braid
    group action of Khovanov--Seidel~\cite{MR1862802}. For us, however, it is
    enough to consider infinitely many braids that are powers of a single Artin
    braid generator. For these complexes it is straightforward to check by hand
    that they are pairwise non-isomorphic. 
\end{proof}

Observe that Theorem~\ref{Thm:Pp} from the introduction is a combination of
Corollary~\ref{cor:SoneBthree}, Proposition~\ref{prop:twopoint}, and
Theorem~\ref{thm:infdim}.

\subsection{Comparison with the Rozansky--Willis invariant}
\label{sec:RW}
In \cite{Rozansky}, Rozansky defined a Khovanov-type homology theory for
(null-homologous) links in $S^1 \times S^2$. His construction was generalized by
Willis in \cite{Willis} to null-homologous links in $Y=\#^m(S^1 \times S^2)$ for
any $m$. We will denote the Rozansky-Willis homology of $L \subset Y$ by
$\RW(L)$. Just like the skein lasagna module $\Sztwo(\Wone, L)$, the invariant
$\RW(L)$ can be computed from a Kirby diagram for $\Wone = \natural^m (S^1 \times
B^3)$ including the link $L$, so it is a natural question whether they are
related.

The first observation is that the two invariants are not always isomorphic.
Indeed, in any specific bidegree, $\RW(L)$ is defined as the Khovanov homology
of the link in $S^3$ obtained from $L$ by adding sufficiently many twists in
place of the \oneh-handles. It follows that $\RW(L)$ has finite rank in each
bidegree, whereas this may not hold for $\Sztwo(\Wone, L)$, as we have seen in
Theorem~\ref{thm:infdim}. Another concrete example is for $m=1$, where
$L=S^1\times P_1$ yields a $4$-dimensional lasagna skein module according to
Proposition~\ref{prop:twopoint}, but $\RW(L)\cong \HH_\bullet(\k[X]/(X^2))$ is
infinite-dimensional.

However, $\RW(L)$ and $\Sztwo(\Wone, L)$ are conceptually similar, as both arise as
the Hochschild homology of a chain complex associated to a tangle $T$ that closes to
the link $L$: 
\begin{itemize}
    \item $\RW(L)$ is computed as the Hochschild homology of a dg bimodule (for
    a tensor product of $m$ of Khovanov's arc rings) associated to the tangle $T$, as
    defined for $m=1$ by Khovanov in \cite{MR1928174} and extended by parabolic
    induction to $m>1$. Here the homological degree of the dg bimodule gets
    mixed with the Hochschild degree, and so the resulting invariant is a
    bigraded vector space.
    \item $\Sztwo(\Wone, L)$ can be computed via Theorem~\ref{thm:mainonehandle}
    (and for $m=1$ even more concretely in Corollary~\ref{cor:HH}) as the
    zeroth Hochschild homology of an equivalent dg bimodule; see
    Remark~\ref{rem:arcrings} for the comparison. In fact, the higher blob
    homology from \cite{MWW}, which does not play a role for skein lasagna modules, corresponds
    to higher Hochschild homology. The main difference, however, is that the dg
    bimodule is not considered as an object of a dg or triangulated category,
    but of the \emph{linear} cohomology category. Accordingly, the full blob
    homology is triply-graded, with the blob/Hochschild grading separated from
    the homological grading.  
\end{itemize} 
Based on this comparison, one may expect $\Sztwo(\Wone, L)$ and, more generally, the
full blob homology $\S_*(\Wone; L)$ to appear on the $E_2$ page of a spectral
sequence converging to $\RW(L)$. Suppose that one can find a suitable projective
resolution in terms of tangle complexes, which simultaneously allows the
computation of blob homology as well as the dg version of Hochschild homology.
Then, by tensoring with the dg bimodule associated to the tangle, one obtains a
double complex of (quantum) graded vector spaces, where the vertical
differential carries Hochschild degree and the horizontal differential carries
homological degree. The homology of the total complex would compute $\RW(L)$. To
obtain $\Sztwo(\Wone, L)$, one first takes homology in the rows (thus computing the
Khovanov homologies of links of the form $T_i\cup \overline{T}$ where $T_i$
appears in the resolution), and only then the zeroth homology of the induced
differential coming from the resolution. We will not pursue this comparison further in the present paper, but
remark that there is precedent for interesting invariants appearing on $E_2$
pages of spectral sequences that come from separating Hochschild and homological
degrees, namely the triply-graded HOMFLYPT link homology; see \cite[Section
6]{MR3447099}.\medskip
 
In general, one does not expect a map from the $E_2$ page of a spectral sequence
to its $E_\infty$ page. However, since $\Sztwo(\Wone, L)$ appears as the lowest row
on the $E_2$ page, the above discussion suggests the existence of a natural map
$$ \Sztwo(\Wone, L) \to \RW(L).$$
In the following we propose a candidate for such a map.

In Willis's construction of $\RW(L)$,
we represent $\del \Wone = Y$ by $m$ pairs of spheres in the plane, with the spheres
in each pair being identified (that is, we add a handle). This is the same as
the usual Kirby diagram of $\Wone$. The link $L$ may intersect each handle a number
of times, as in this picture:
\[ %% Creator: Inkscape 1.0.2 (e86c8708, 2021-01-15), www.inkscape.org
%% Accompanies image file 'L.pdf' (pdf, eps, ps)
\begingroup%
  \makeatletter%
  \providecommand\color[2][]{%
    \errmessage{(Inkscape) Color is used for the text in Inkscape, but the package 'color.sty' is not loaded}%
    \renewcommand\color[2][]{}%
  }%
  \providecommand\transparent[1]{%
    \errmessage{(Inkscape) Transparency is used (non-zero) for the text in Inkscape, but the package 'transparent.sty' is not loaded}%
    \renewcommand\transparent[1]{}%
  }%
  \providecommand\rotatebox[2]{#2}%
  \newcommand*\fsize{\dimexpr\f@size pt\relax}%
  \newcommand*\lineheight[1]{\fontsize{\fsize}{#1\fsize}\selectfont}%
  \ifx\svgwidth\undefined%
    \setlength{\unitlength}{179.8843578bp}%
    \ifx\svgscale\undefined%
      \relax%
    \else%
      \setlength{\unitlength}{\unitlength * \real{\svgscale}}%
    \fi%
  \else%
    \setlength{\unitlength}{\svgwidth}%
  \fi%
  \global\let\svgwidth\undefined%
  \global\let\svgscale\undefined%
  \makeatother%
  \begin{picture}(1,0.17207847)%
    \lineheight{1}%
    \setlength\tabcolsep{0pt}%
    \put(0.00786129,0.12929934){\makebox(0,0)[lt]{\lineheight{1.25}\smash{\begin{tabular}[t]{l}$L$\end{tabular}}}}%
    \put(0,0){\includegraphics[width=\unitlength,page=1]{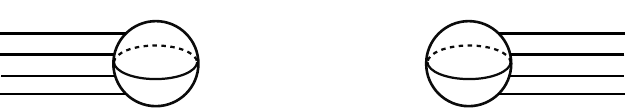}}%
  \end{picture}%
\endgroup%
 \] Let $L(n_1, \dots, n_m)$ be the link in $S^3$ obtained
from $L$ by inserting $n_i$ full twists in place of the $i^{\operatorname{th}}$
handle, as shown here:
\[ %% Creator: Inkscape 1.0.2 (e86c8708, 2021-01-15), www.inkscape.org
%% Accompanies image file 'Lni.pdf' (pdf, eps, ps)
\begingroup%
  \makeatletter%
  \providecommand\color[2][]{%
    \errmessage{(Inkscape) Color is used for the text in Inkscape, but the package 'color.sty' is not loaded}%
    \renewcommand\color[2][]{}%
  }%
  \providecommand\transparent[1]{%
    \errmessage{(Inkscape) Transparency is used (non-zero) for the text in Inkscape, but the package 'transparent.sty' is not loaded}%
    \renewcommand\transparent[1]{}%
  }%
  \providecommand\rotatebox[2]{#2}%
  \newcommand*\fsize{\dimexpr\f@size pt\relax}%
  \newcommand*\lineheight[1]{\fontsize{\fsize}{#1\fsize}\selectfont}%
  \ifx\svgwidth\undefined%
    \setlength{\unitlength}{184.12322457bp}%
    \ifx\svgscale\undefined%
      \relax%
    \else%
      \setlength{\unitlength}{\unitlength * \real{\svgscale}}%
    \fi%
  \else%
    \setlength{\unitlength}{\svgwidth}%
  \fi%
  \global\let\svgwidth\undefined%
  \global\let\svgscale\undefined%
  \makeatother%
  \begin{picture}(1,0.15885591)%
    \lineheight{1}%
    \setlength\tabcolsep{0pt}%
    \put(0.47200511,0.06281895){\makebox(0,0)[lt]{\lineheight{1.25}\smash{\begin{tabular}[t]{l}$n_i$\end{tabular}}}}%
    \put(0,0){\includegraphics[width=\unitlength,page=1]{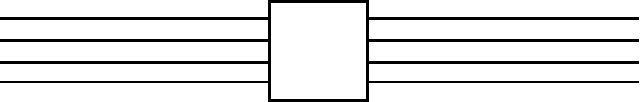}}%
  \end{picture}%
\endgroup%
 \] The homology $\RW(L)$ can be computed as the Khovanov
homology of the link $L(n_1, \dots, n_m)$ for $n_i \gg 0$, with some suitable
shifts in grading. Note that $L(n_1, \dots, n_m)$ depends on the choice of a
path between the attaching spheres of each \oneh-handle; however, it can be
shown that $\RW(L)$ is independent of these choices up to isomorphism.

Consider now the skein lasagna module $\Sztwo(\Wone, L)$. Let us attach an
$n_i$-framed \twoh-handle through the $i^{\operatorname{th}}$ \oneh-handle:
\[ %% Creator: Inkscape 1.0.2 (e86c8708, 2021-01-15), www.inkscape.org
%% Accompanies image file 'Lhandle.pdf' (pdf, eps, ps)
\begingroup%
  \makeatletter%
  \providecommand\color[2][]{%
    \errmessage{(Inkscape) Color is used for the text in Inkscape, but the package 'color.sty' is not loaded}%
    \renewcommand\color[2][]{}%
  }%
  \providecommand\transparent[1]{%
    \errmessage{(Inkscape) Transparency is used (non-zero) for the text in Inkscape, but the package 'transparent.sty' is not loaded}%
    \renewcommand\transparent[1]{}%
  }%
  \providecommand\rotatebox[2]{#2}%
  \newcommand*\fsize{\dimexpr\f@size pt\relax}%
  \newcommand*\lineheight[1]{\fontsize{\fsize}{#1\fsize}\selectfont}%
  \ifx\svgwidth\undefined%
    \setlength{\unitlength}{179.8843578bp}%
    \ifx\svgscale\undefined%
      \relax%
    \else%
      \setlength{\unitlength}{\unitlength * \real{\svgscale}}%
    \fi%
  \else%
    \setlength{\unitlength}{\svgwidth}%
  \fi%
  \global\let\svgwidth\undefined%
  \global\let\svgscale\undefined%
  \makeatother%
  \begin{picture}(1,0.17207847)%
    \lineheight{1}%
    \setlength\tabcolsep{0pt}%
    \put(0.00786129,0.12929934){\makebox(0,0)[lt]{\lineheight{1.25}\smash{\begin{tabular}[t]{l}$L$\end{tabular}}}}%
    \put(0,0){\includegraphics[width=\unitlength,page=1]{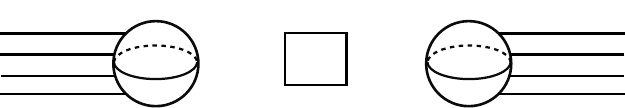}}%
    \put(0.47615324,0.05865839){\makebox(0,0)[lt]{\lineheight{1.25}\smash{\begin{tabular}[t]{l}$n_i$\end{tabular}}}}%
    \put(0,0){\includegraphics[width=\unitlength,page=2]{Lhandle.pdf}}%
  \end{picture}%
\endgroup%
 \] The \twoh-handles cancel the corresponding
\oneh-handles, so the result is a Kirby diagram for $B^4$, whose boundary is
$S^3$. The link $L$ becomes $L(n_1, \dots, n_m) \subset S^3$, as can be seen by
doing a series of handle slides of the arcs of $L$ over the \twoh-handle:
\[ %% Creator: Inkscape 1.0.2 (e86c8708, 2021-01-15), www.inkscape.org
%% Accompanies image file 'Lhandleslide.pdf' (pdf, eps, ps)
\begingroup%
  \makeatletter%
  \providecommand\color[2][]{%
    \errmessage{(Inkscape) Color is used for the text in Inkscape, but the package 'color.sty' is not loaded}%
    \renewcommand\color[2][]{}%
  }%
  \providecommand\transparent[1]{%
    \errmessage{(Inkscape) Transparency is used (non-zero) for the text in Inkscape, but the package 'transparent.sty' is not loaded}%
    \renewcommand\transparent[1]{}%
  }%
  \providecommand\rotatebox[2]{#2}%
  \newcommand*\fsize{\dimexpr\f@size pt\relax}%
  \newcommand*\lineheight[1]{\fontsize{\fsize}{#1\fsize}\selectfont}%
  \ifx\svgwidth\undefined%
    \setlength{\unitlength}{184.24971317bp}%
    \ifx\svgscale\undefined%
      \relax%
    \else%
      \setlength{\unitlength}{\unitlength * \real{\svgscale}}%
    \fi%
  \else%
    \setlength{\unitlength}{\svgwidth}%
  \fi%
  \global\let\svgwidth\undefined%
  \global\let\svgscale\undefined%
  \makeatother%
  \begin{picture}(1,1.37053126)%
    \lineheight{1}%
    \setlength\tabcolsep{0pt}%
    \put(0.44412369,0.78009397){\makebox(0,0)[lt]{\lineheight{1.25}\smash{\begin{tabular}[t]{l}. . . \end{tabular}}}}%
    \put(0,0){\includegraphics[width=\unitlength,page=1]{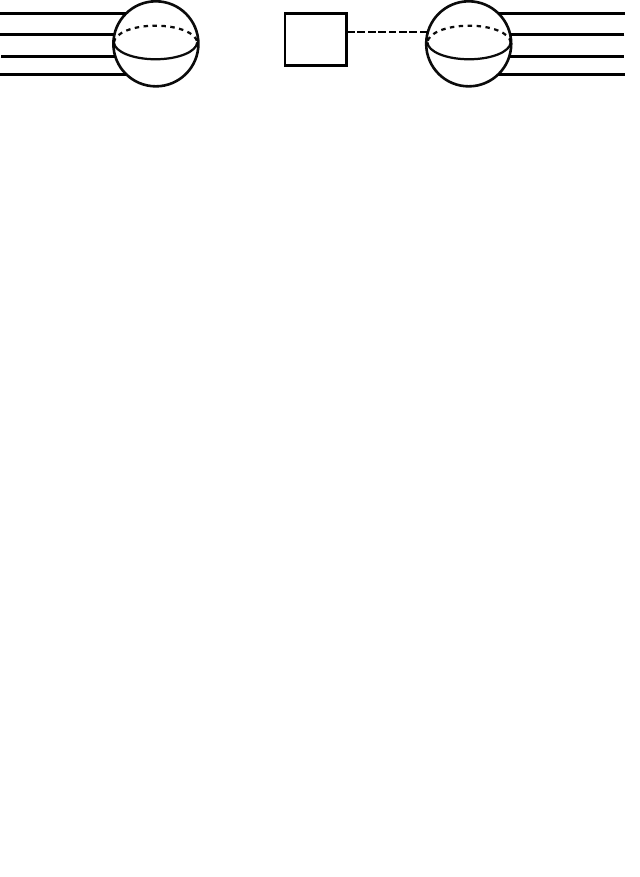}}%
    \put(0.46487193,1.29079397){\makebox(0,0)[lt]{\lineheight{1.25}\smash{\begin{tabular}[t]{l}$n_i$\end{tabular}}}}%
    \put(0,0){\includegraphics[width=\unitlength,page=2]{Lhandleslide.pdf}}%
    \put(0.46487545,0.99063497){\makebox(0,0)[lt]{\lineheight{1.25}\smash{\begin{tabular}[t]{l}$n_i$\end{tabular}}}}%
    \put(0,0){\includegraphics[width=\unitlength,page=3]{Lhandleslide.pdf}}%
    \put(0.46629172,0.46052723){\makebox(0,0)[lt]{\lineheight{1.25}\smash{\begin{tabular}[t]{l}$n_i$\end{tabular}}}}%
    \put(0,0){\includegraphics[width=\unitlength,page=4]{Lhandleslide.pdf}}%
    \put(0.46432566,0.06082728){\makebox(0,0)[lt]{\lineheight{1.25}\smash{\begin{tabular}[t]{l}$n_i$\end{tabular}}}}%
  \end{picture}%
\endgroup%
 \] where in the last step we cancelled the
handles. (Compare Figure 5.13 in \cite{GS}.)

The \twoh-handle attachments give a cobordism $Z$ from $Y=\#^m(S^1 \times S^2)$ to
$S^3$. There is also an embedded annular cobordism $S\subset Z$ from $L$ to
$L(n_1, \dots, n_m)$. As discussed in Section~\ref{sec:gluing}, these cobordisms
induce a map on skein lasagna modules:
$$\Psi_{Z;S} : \Sztwo(\Wone; L) \to \Sztwo(B^4; L(n_1, \dots, n_m)) \cong \Kh(L(n_1,
\dots, n_m)).$$ Our conjecture is that these maps stabilize as $n_i \to \infty$,
giving a well-defined morphism from $\Sztwo(\Wone, L)$ to $\RW(L)$.

\subsection{Speculations on homotopy coherent \fourm-manifold invariants}
\label{sec:spec}
We expect that the above $E_2$-page-of-spectral-sequence relationship between
$\Sztwo$ and $H_{RW}^{*,*}$ for $(\natural^m(S^1\times B^3), L)$
generalizes to $(\Wgeneric,L)$ for arbitrary \fourm-manifolds $\Wgeneric$ and links $L$. We give a brief sketch of
the reasoning below.

Recall that the Khovanov-Rozansky invariants upon which $\Sz$ is built
assign chain complexes to links $L$ and chain maps to link cobordisms, but it is
not known that this assignment is functorial (or even well-defined) at the level
of complexes.  The proof that the homology of these complexes is functorial in
the appropriate sense involves showing that certain chain maps are homotopic.
If this result could be strengthened to show that certain homotopies between the
chain maps are themselves 2nd-order homotopic, and so on for all higher orders,
then one could construct a functorial assignment of chain complexes to links in
$S^3$ and chain maps to link cobordisms.

Let us assume that these conjectured ``fully coherent" $\glN$ chain complexes for links exist. Then, they can be
repackaged as a pivotal $(\infty,4)$-category (with composition maps defined in
terms of link cobordisms, as in \cite{MWW}).  This $(\infty,4)$-category can in
turn be fed into the machinery of Section 6.3 of \cite{MWblob} (which is closely
related to topological chiral homology \cite{lurie-higher-algebra} and
factorization homology \cite{MR3431668,MR3595895}).  The result is a
chain-complex-valued invariant $\Sinf(\Wgeneric,L)$. Its construction involves taking a homotopy colimit of a
poset built out of the set of all ball decompositions of $\Wgeneric$ and refinement
relationships between these ball decompositions.  Concretely, we construct a
double complex, with horizontal differentials coming from the $\glN$ complexes of links,
and vertical differentials coming from the combinatorics of refining ball
decompositions of $\Wgeneric$. There is a spectral sequence associated to this double
complex, which is itself an invariant of $(\Wgeneric,L)$. 

The $E_2$ page of this spectral sequence involves first taking homology in the
horizontal direction, then computing homology with respect to vertical
differentials.  It is easy to see that this $E_2$ page is exactly the blob
homology $\S_*(\Wgeneric; L)$ assigned to $(\Wgeneric,L)$ in \cite{MWW} (i.e. by taking $\KhR$ homology early
instead of working with the $\glN$ complex).  (In this paper we have focused on
blob-degree zero, corresponding to the bottom row of the $E_2$ page of the
spectral sequence.)

When $\Wone = \natural^m(S^1\times B^3)$ and $N=2$, we expect the total homology of
$\Sinf(\Wone,L)$ to coincide with the Rozansky--Willis invariants. The Hochschild
differentials of the previous subsection should be (homotopy equivalent to)
special cases of the vertical differentials above.

 %This file deals with the skein module in 1-handles

% \bibliographystyle{custom}
% \bibliography{biblio}

\end{document}